\documentclass[a4paper, 10pt, DIV = 13]{scrartcl}
\usepackage[utf8]{inputenc}
\usepackage[english]{babel}

\usepackage{amsmath,amsfonts,amssymb,amsthm,mathrsfs,todonotes}
\usepackage{amsfonts}
\usepackage{amssymb}
\usepackage{wasysym}
\usepackage{lmodern}
\usepackage{hyperref}
\usepackage{xcolor}
\usepackage{hyperref}
\usepackage[normalem]{ulem}
\hypersetup{breaklinks=true}

\usepackage{mathtools}
\mathtoolsset{showonlyrefs}

\definecolor{darkred}{rgb}{0.5,0,0}
\definecolor{darkgreen}{rgb}{0,0.5,0}
\definecolor{darkblue}{rgb}{0,0,0.5}
\hypersetup{colorlinks,
linkcolor=darkblue,
filecolor=darkgreen,
urlcolor=darkred,
citecolor=darkblue}

\numberwithin{equation}{section}

\newtheorem{proposition}{Proposition}[section]
\newtheorem{lemma}[proposition]{Lemma}
\newtheorem{corollary}[proposition]{Corollary}

\newtheorem{claim}[proposition]{Claim}

\theoremstyle{definition}
\newtheorem{theorem}{Theorem}
\newtheorem*{theorem*}{Theorem}
\newtheorem{definition}[proposition]{Definition}
\newtheorem{remark}[proposition]{Remark}

\usepackage{enumitem}
\setlist{nosep}
\setlist{noitemsep}
\setlist{leftmargin=*}

\usepackage{mathtools}

\DeclarePairedDelimiter\floor{\lfloor}{\rfloor}

\def\XXint#1#2#3{{\setbox0=\hbox{$#1{#2#3}{\int}$}
     \vcenter{\hbox{$#2#3$}}\kern-.5\wd0}}

\newcommand{\R}{\mathbb{R}}
\newcommand{\Z}{\mathbb{Z}}

\renewcommand{\epsilon}{\varepsilon}
\newcommand{\hal}{\frac{1}{2}}

\newcommand{\dd}{\mathtt{d}}

\newcommand{\E}{\mathbb{E}}
\newcommand{\I}{\mathbf{1}}

\newcommand{\Jd}{\mathrm{J}_{\d/2}}

\newcommand{\dist}{\mathrm{dist}}

\newcommand{\La}{\mathcal{L}}
\newcommand{\DD}{\mathcal{D}}
\newcommand{\Cov}{\mathrm{Cov}}
\newcommand{\Wass}{\mathrm{Wass}}

\newcommand {\DF}[1]{\todo[inline,size=\footnotesize,color=cyan]{\textbf{DF:} #1}}

\begin{document}
\title{\LARGE{(Non)-hyperuniformity of perturbed lattices}}
\author{\large{David Dereudre\thanks{University of Lille, LPP UMR 8524, \texttt{david.dereudre@univ-lille.fr}}, Daniela Flimmel\thanks{Charles University,
\texttt{daniela.flimmel@karlin.mff.cuni.cz}}, Martin Huesmann\thanks{Institute for Mathematical Stochastics, University of Münster, \texttt{martin.huesmann@uni-muenster.de}}, Thomas Leblé\thanks{ Université de Paris-Cité, CNRS, MAP5 UMR 8145, F-75006 Paris, France \texttt{thomas.leble@math.cnrs.fr} }}}

\newcommand{\bX}{\mathbf{X}}
\newcommand{\X}{\mathrm{X}}
\renewcommand{\d}{\mathsf{d}}
\newcommand{\Rd}{\R^\d}
\newcommand{\Zd}{\Z^\d}
\newcommand{\B}{\mathrm{B}}
\newcommand{\BR}{\B_r}
\newcommand{\Var}{\mathrm{Var}}
\newcommand{\SL}{\mathbf{L}}
\newcommand{\PSL}{\mathbf{PL}}
\renewcommand{\P}{\mathbb{P}}
\newcommand{\p}{\mathrm{p}}
\newcommand{\bp}{\bar{\p}}

\date{\small{\today}}
\maketitle
\begin{abstract}
We ask whether a stationary lattice in dimension $\d$ whose points are shifted by identically distributed but possibly dependent perturbations remains hyperuniform. {When $\d = 1$ or $2$, we show that it is the case when the perturbations have a finite $\d$-moment,} and that this condition is sharp. When $\d \geq 3$, we construct arbitrarily small perturbations such that the resulting point process is not hyperuniform. {As a side remark of independent interest, we exhibit hyperuniform processes with arbitrarily slow decay of their number variance.}
\end{abstract}

\renewcommand{\O}{\mathcal{O}}

\newcommand{\kk}{\mathsf{k}}
\newcommand{\1}{\mathsf{1}}
\newcommand{\2}{\mathsf{1}}
\newcommand{\Hess}{\mathrm{Hess}}

\section{Introduction}
\label{sec:intro}
\subsection{Setting and main results}
Let $\d \geq 1$ be the ambient dimension. A \emph{point configuration} is a locally finite collection of points in $\Rd$, possibly with multiplicities, which we identify with a purely atomic Radon measure $\bX$ (see Section \ref{sec:pointprocesses} for technical details). For $r > 0$, we let $\BR$ be the closed Euclidean ball of radius $r$ centered at the origin and if $\bX$ is a point configuration we denote by $|\bX \cap \BR|$  the number of its points in $\BR$. The following definition has received a lot of interest in the past twenty years:

\vspace{0.2cm}
Let $\bX$ be a \emph{random} point configuration (i.e. a \emph{point process}). We say that $\bX$ is \emph{hyperuniform} when:
\begin{equation}
\label{def:HU}
{\sigma(r) := \frac{\Var\left[|\bX \cap \BR|\right]}{|\BR|} \longrightarrow 0} \text{ as } r \to \infty,
\end{equation}
where $|\BR|$ denotes the volume of a ball of radius $r$. The quantity $\Var[|\bX \cap \BR|]$ is called the \emph{number variance} (of the random point configuration $\bX$, in $\BR$) {and $\sigma$ the rescaled number variance}.
Of course, the volume $|\BR|$ is proportional to $r^\d$ so when \eqref{def:HU} holds one can wonder how slow the number variance growth actually is compared to $r^\d$. In the terminology of \cite{torquato2018hyperuniform}, $\bX$ is said to be \emph{hyperuniform of class I} when in fact:
\begin{equation}
\label{typeI}
\Var\left[|\bX \cap \BR|\right] = \O(r^{\d-1}) \text{ as } r \to \infty.
\end{equation}
The basic example of a class I-hyperuniform point process is the “stationary lattice” $\SL : = \{x + U, \ x \in \La\}$, where $\La$ is some lattice in $\Rd$ (i.e.\ $\Zd$ or the image of $\Zd$ by an affine transformation), and the random variable $U$ is uniformly distributed in a fundamental domain of $\La$. Since \eqref{typeI} is the slowest possible growth (cf. \cite[Thm.\ 2A]{Be87}), stationary lattices are “as hyperuniform as it gets”. 

In this paper, we consider \emph{perturbed lattices} of the form $\PSL := \{ x + U + \p_x, \ x \in \La\}$ where $U$ is as above while $(\p_x)_{x \in \La}$ is a family of random variables independent of $U$, thought of as “perturbations”, with values in $\Rd$, and we discuss whether $\PSL$ remains hyperuniform or not. 

When the perturbations are independent and identically distributed, the answer is well-known: \emph{in any dimension}, if $\p_0$ has a \emph{finite first moment}, then $\PSL$ remains \emph{hyperuniform of class I}, see \cite{gacs1975problem}. In fact, even if $\p_0$ has only a finite moment of order $0 < s < 1$, then $\PSL$ remains hyperuniform (but not necessarily of class I). 

For dependent perturbations, with (or without) information about the structure of dependence (as discussed e.g. in \cite{gabrielli2004point,Fli2025,KLLY}), the question was mainly open.

Throughout this paper, we make the following assumption: the perturbations $(\p_x)_{x \in \La}$ are identically distributed, and moreover their joint distribution is $\La$-invariant in the sense that for all $x_\star$ in $\La$, $(\p_x)_{x \in \La}$ and $(\p_{x+x_\star})_{x \in \La}$ have the same  distribution. It implies that the perturbed lattice $\PSL$ itself is stationary. 

We may now state our main results. Let us emphasize that \emph{we allow the perturbations to be dependent}.
\begin{theorem}[(Non)-hyperuniformity of perturbed lattices]  
\label{theo:main} 
\

\begin{itemize}
   \item For all $\d \geq 1$, if $\E[|\p_0|^\d]$ is finite, then the number variance of $\PSL$ in $\BR$ is finite for every $r > 0$.
   \item For $\d = 1$ or $2$, if $\E[|\p_0|^\d]$ is finite, then $\PSL$ is hyperuniform in the sense of \eqref{def:HU}.
   \item For all $\d \geq 1$, taking $\La = \Zd$, if $X$ is a positive random variable such that $\E[X^\d] = + \infty$, then there exist perturbations $(\p_x)_{x \in \Zd}$, with $|\p_0| \leq X$ almost surely, such that the resulting perturbed lattice has infinite number variance in finite balls. 
   \item For all $\d \geq 3$, taking $\La = \Zd$, for all $\epsilon > 0$ there exist perturbations  $(\p_x)_{x \in \Zd}$ with $|\p_0| \leq \epsilon$ almost surely, such that the resulting perturbed lattice satisfies $\lim_{r \to + \infty} \frac{\Var\left[|\PSL \cap \BR|\right]}{|\BR|} = + \infty$.
\end{itemize}
\end{theorem}
The third item shows that in dimension $1$ and $2$ the sufficient condition of the second item allowing to retain hyperuniformity is particularly sharp, because anything below that threshold can give rise to processes which not only are \emph{not hyperuniform}, but have in fact \emph{infinite number variance in finite balls}. In higher dimensions ($d \geq 3$) the last item shows that arbitrarily small perturbations can break hyperuniformity.

\begin{theorem}[{About class-I hyperuniformity}]
\label{theo:classI}
\

\begin{itemize}
   \item For $\d = 1$, if $\E[|\p_0|^2]$ is finite, then $\PSL$ is class~I hyperuniform (which, in dimension $1$, means that the number variance stays bounded as the size of the interval goes to $+ \infty$). This condition is sharp. 
   \item For $\d = 2$, there are perturbations of $\Z^2$ which are almost surely bounded by an arbitrarily small constant, and for which the perturbed lattice is not class~I hyperuniform. 
\end{itemize}
\end{theorem}

We thus obtain with the first item a sharp sufficient condition for class-I hyperuniformity of perturbed lattices in dimension $1$ based only on the size of the perturbations. The sharpness means that we can construct a perturbation such that $\E[|\p_0|^2]=+\infty$ and $\PSL$ is not class~I hyperuniform, see Remark \ref{rem:sharpd1}. In dimension $2$, the second item indicates that any sufficient condition should presumably involve the correlation structure of the perturbations, see Section \ref{sec:open}.

\newcommand{\tsigma}{\tilde{\sigma}}
\begin{theorem}[Arbitrarily slow decay of the number variance]
\label{theo:slow}
Let $\tsigma : (0, + \infty) \to (0, + \infty)$ be a non-increasing positive function with $\lim_{r \to + \infty} \tsigma(r) = 0$. For all $\d \geq 1$, there exist $c > 0$ and a perturbed lattice $\PSL$ such that:
\begin{enumerate}
   \item $\PSL$ is hyperuniform.
   \item The rescaled number variance $\sigma$ of $\PSL$ (see \eqref{def:HU}) satisfies:
   \begin{equation}
   \label{eq:sigmatsigma}
      \sigma(r) \geq c \tsigma(r) \text{ for all } r \geq 1.
   \end{equation}
\end{enumerate}
\end{theorem}
We emphasize that $\tsigma$ can really be chosen tending to $0$ as slow as one wants.

\subsection{About hyperuniformity}
\label{sec:HU}

\subsubsection*{General facts}
In the statistical physics literature of the 1980's (see e.g. \cite{MY80,L83}) it has been observed that some classical systems of particles with long-range interactions exhibit \emph{cancellation of charge fluctuations}, namely that the variance of the number of charges in a domain increases proportionally to the surface area of the domain's boundary, and not to the volume. The notion of \emph{hyperuniformity}, or equivalently of \emph{super-homogeneous} distributions of points was introduced in \cite{gabrielli2002glass,torquato2003local}, with applications in theoretical chemistry and physics. We refer to the survey by S. Torquato \cite{torquato2018hyperuniform} for an overview of the considerable literature that has since been devoted to this topic.

One motivation to study hyperuniform processes is that they can exhibit a form of “order within disorder”: the fact that their rescaled number variance goes to $0$ can be interpreted as a form of collective order yet their truncated correlation functions may all decay very fast at infinity hence there is no “long-range order” in the usual sense of statistical physics. It is worth mentioning that since the “number of points in big boxes” is a very natural observable to study, some “hyperuniformity” results can predate the coining of the term itself by decades. 

Amusingly, the original definition for hyperuniformity was more stringent than \eqref{def:HU} and corresponds to what is now called “class-I hyperuniformity” (see \eqref{typeI}), as it has been realized that many interesting processes satisfy \eqref{def:HU} but not \eqref{typeI}. The defining condition \eqref{def:HU} should be understood by comparison to the fundamental example of a random point configuration given by the Poisson point process. Recall that for a Poisson point process $\bX$ (with constant intensity $c$), the number of points $|\bX \cap \BR|$ follows a Poisson distribution of parameter $c|\BR|$ for all $r > 0$, and as a straightforward consequence the number variance is precisely proportional to the volume $|\BR|$. For a random point configuration $\bX$ to be hyperuniform thus means that its number variance (in large balls) is asymptotically “much smaller than for a Poisson point process”.

\subsubsection*{Hyperuniformity and structure factor}
Assume that the (stationary) point process $\bX$ {has intensity $1$ (i.e.\ one point per unit volume on average) and} is nice enough so that the following quantity:
\begin{equation*}
g_2(x) := \lim_{\epsilon \to 0} \frac{1}{|\B_\epsilon|^2} \P[\bX \text{ has at least one point in } \B_\epsilon(0) \text{ and in } \B_\epsilon(x)]
\end{equation*}
exists for all $x \in \Rd$. The map $x \mapsto g_2(x)$ is usually called the \emph{pair correlation function} of $\bX$. The \emph{structure factor} $S$ is then defined as: $S := 1 + \widehat{(g_2 - 1)}$, where $\widehat{\cdot}$ represents the Fourier transform (see Section \ref{Sec:Notations} for notation and conventions).

Heuristically speaking, there is a correspondence between the small wave-number behavior of the structure factor (i.e.\ $S(k)$ for $|k|$ small) and the growth of number variance in large balls (see for instance \cite[Sec. 2]{torquato2018hyperuniform}), {which can be summarized as follows}:
\begin{itemize}
   \item $S(0) = 0$ $\longleftrightarrow$ $\Var\left[|\bX \cap \BR|\right] = o(r^\d)$ ($\bX$ is hyperuniform). This is sometimes phrased as the condition “$\lim_{k \to 0} S(k) = 0$”, or as the \emph{sum rule} $\int_{\Rd} (g_2 - 1) = -1$.
   \item $S(k) \sim |k|^{\alpha}$ with $\alpha > 1$ $\longleftrightarrow$ $\Var\left[|\bX \cap \BR|\right] = \O(r^{\d-1})$ ($\bX$ is class-I hyperuniform)
   \item $S(k) \sim |k|^{\alpha}$ with $\alpha \in (0, 1)$ $\longleftrightarrow$ $\Var\left[|\bX \cap \BR|\right] = \O(r^{\d-\alpha})$ ($\bX$ is hyperuniform, but not Class-I)
\end{itemize}
On a mathematical level, there are several aspects to care about:
\begin{enumerate}
   \item The two-point correlation function is in general not a proper function and has to be defined as a measure (we recall this in Section \ref{sec:pointprocesses} below). The measure $g_2 - \dd x$ might have infinite mass, so its Fourier transform (and thus the structure factor itself) has to be defined in a distributional sense.
   \item The value of $S$ at $0$ might not be well-defined, or $S$ might not be continuous near $0$. As a matter of fact, it can happen that the statement $\lim_{|k| \to 0} S(k) = 0$ is \emph{false} (because the limit does not exist) and yet the point process \emph{is hyperuniform}. This is not a big issue in itself - in actual computations related to hyperuniformity, the structure factor always gets integrated against some “nice” functions.
\end{enumerate}
The mathematical connection between the \emph{geometric} definition \eqref{def:HU} of hyperuniformity and its \emph{spectral} characterization was first fully established in \cite{bjorklund2024hyperuniformity}. In particular, it is proven that: 
\begin{proposition}[\cite{bjorklund2024hyperuniformity}, Thm. 3.6 \& 3.7]
\label{prop:BH}
Let $S$ be the structure factor of stationary point process.
\begin{enumerate}
   \item If $S(\B_{\epsilon}) = o(\epsilon^\d)$ as $\epsilon \to 0$ (“spectral hyperuniformity”), then $\sigma(r) = o(r^\d)$ (“geometric hyperuniformity”).
   \item If $S(\B_{\epsilon}) = \O\left(\epsilon^{\d+\gamma}\right)$ then $\sigma(r) = \O\left(r^{\d - \gamma}\right)$.
\end{enumerate}
\end{proposition}

\subsection{Connection with the literature}
\subsubsection*{Positive result}
Our main positive result in Theorem \ref{theo:main} is the fact that in dimension $2$, perturbations with finite second moment preserve the hyperuniformity of lattices. To the best of our knowledge, this was not proven in the mathematical literature \emph{even for bounded perturbations}. In the physics literature, this is discussed in several places including \cite[Sec. VI.C]{gabrielli2002glass}, \cite{gabrielli2004point}, \cite[Sec. IV. A]{kim2018effect} - all of which do state that hyperuniformity of a two-dimensional lattice cannot be broken if the perturbations have finite variance. Their argument relies on computing a small wave-number approximation for the structure factor of $\PSL$, in which the covariance structure of the perturbations plays an important role. This is detailed in \cite{gabrielli2004point}, which includes the case of Gaussian perturbations ({note that the paper also considers the perturbations of general hyperuniform processes, not only stationary lattices}).


\subsubsection*{Examples of pathological processes}
Our examples (see again Theorem \ref{theo:main}) of non-hyperuniform perturbed lattices in dimension $\d = 2$ (with perturbations failing to have a finite second moment) and $\d \geq 3$ (with arbitrarily small perturbations) seem to be the first mathematical constructions of such objects. Their existence was hinted on in the physics papers cited above (e.g. in \cite{kim2018effect} \textit{“In contrast to uncorrelated perturbed lattices that are always hyperuniform, correlated ones can be non-hyperuniform.”}), but no explicit construction was given - rather, based on the small wave-number expansion for the structure factor mentioned above, they argue that if the spectral density of the perturbation field has a certain power-law near $0$, then hyperuniformity will be broken. In dimension $2$, this does not suggest how to proceed (because in that case, the perturbations must have infinite variance, hence no properly defined spectral density). In higher dimensions, this suggests that certain Gaussian perturbation fields with well-chosen spectral densities can do the job. We produce counter-examples with perturbations that are \emph{bounded} and not “only” with finite moments of all order like Gaussian variables. 

Our constructions are elementary and allow us to prove lower bounds on the number variance directly, without relying on the structure factor.  

\subsubsection*{Slow decay of the rescaled number variance}
In Theorem \ref{theo:slow}, we give what seems to be the first examples of hyperuniform processes in any dimension with truly arbitrarily slow decay of the rescaled number variance (although this might have been “folklore” knowledge to some experts, the only general result that we are aware of is the fact that in dimension $1$, any power law decay can be realized by a well-chosen determinantal point process, see \cite[Sec. 3]{soshnikov2000determinantal}). For his phenomenological classification, Torquato \cite{torquato2018hyperuniform} considers three asymptotic rates: class-I ($\sigma(r) \sim c r^{-1}$), class-II ($\sigma(r) \sim c r^{-1} \log r$) and class-III ($\sigma(r) \sim c r^{-\alpha}$ for some $\alpha \in (0,1)$), which does \emph{not} include cases like $\sigma(r) \geq c \log^{-0.01} r$. We thus show that this classification does not quite cover the entire spectrum of possible behaviors. It is however fair to say that our examples are ad hoc constructions with little “physical” reality.

\subsection{An optimal transport perspective}
In this section, we reformulate our main result from the point of view of optimal transport between stationary point processes recently introduced in \cite{ErHuJaMu23} {(using their notation)}. 

Let us recall the relevant notions. Denote the set of locally finite point configurations by $\Gamma$. For two distributions of stationary point processes $\mathsf{P}_0$ and $\mathsf{P}_1$ and a cost function $c:\Gamma\times\Gamma\to\R$ we are interested in the transport problem
\begin{equation}
\label{eq:generalTransport}
\inf_{Q\in \mathsf{Cpl}_s(\mathsf{P}_0,\mathsf{P}_1)} \int c \d Q
\end{equation}
where $\mathsf{Cpl}_s(\mathsf{P}_0,\mathsf{P}_1)$ is the set of all couplings between $\mathsf{P}_0$ and $\mathsf{P}_1$ that are invariant under the diagonal action of $\R^d$ on $\Gamma\times\Gamma$ given by $\theta_x(\xi,\eta)(\cdot,\cdot)=(\xi,\eta)(\cdot -x,\cdot -x)$ {(for all $x \in \R^d$)}. 

Note that in \eqref{eq:generalTransport}, $c$ is a cost function \emph{between point configurations}. A particularly relevant choice of $c$ is the following: for $p\geq 1$ and for any two point configurations $\xi, \eta$ on $\Rd$, we put:
\begin{equation*}
c_p(\xi,\eta) :=\inf_{q\in\mathsf{cpl}(\xi,\eta)}\limsup_{n\to\infty}\frac{1}{n^d} \int_{\Lambda_n\times\R^d}|x-y|^p q(dx,dy),
\end{equation*}
where $\mathsf{cpl}(\xi,\eta)$ denotes the set of all couplings between the point configurations $\xi$ and $\eta$, seen as measures on $\Rd$ and $\Lambda_n$ is the box $[-n/2,n/2]^d$. The quantity $c_p(\xi, \eta)$ can be thought of as a “$p$-Wasserstein distance per unit volume” between the measures $\xi$ and $\eta$. 

Then, if $\mathsf{P}_0,\mathsf{P}_1$ are the distributions of two random point configurations, we define their “$p$-Wasserstein distance” $\Wass_p$ by: 
\begin{equation*}
\Wass_p^p(\mathsf{P}_0,\mathsf{P}_1) :=\inf_{Q\in \mathsf{Cpl}_s(\mathsf{P}_0,\mathsf{P}_1)} \int c_p \d Q.
\end{equation*}
This defines a geodesic metric on the space of stationary point process with unit intensity.

{In order to connect this definition of a Wasserstein distance between point processes with the result of the present paper, it is relevant to mention that} there is a representation of $\Wass_p$ in terms of equivariant random matchings. Let $(\xi^\bullet,\eta^\bullet)$ be {two random point configurations whose distribution is} jointly invariant {under the shifts $\theta_x$ (for $x \in \R^d$)}, denote their marginal distributions by $\mathsf{P}_0$, $\mathsf{P}_1$. {Assume that $(\xi^\bullet,\eta^\bullet)$ are constructed} on some probability space $(\Omega, \mathcal F,\mathbb P)$ equipped with a measurable flow $\theta:\Omega\to\Omega$ satisfying $\theta_0=id, \theta_x\circ\theta_y=\theta_{x+y}.$
Denote by $\mathsf{cpl}_e(\xi^\bullet,\eta^\bullet)$ the set of all equivariant couplings of $\xi^\bullet$ and $\eta^\bullet$, i.e.\ for each $\omega\in\Omega$ the measure $q^\omega$ is a coupling of $\xi^\omega$ and $\eta^\omega$ and $q^{\theta_x\omega}(A,B)=q^\omega(A-x,B-x)$. Then, we have  \cite[Prop. 1.2]{ErHuJaMu23}
 \begin{equation}\label{eq:equicost}
 \Wass_p^p(\mathsf{P}_0,\mathsf{P}_1)=\inf_{(\xi^\bullet,\eta^\bullet)}\inf_{q^\bullet\in\mathsf{cpl}_e(\xi^\bullet,\eta^\bullet)} \E\left[\int_{[0,1]^\d \times\R^d}|x-y|^p q(dx,dy)\right],
 \end{equation}
where the first infimum runs over all probability spaces supporting jointly invariant random measures $\xi^\bullet\sim\mathsf{P}_0,\eta^\bullet\sim\mathsf{P}_1.$ Furthermore, the infimum is attained by an invariant matching $\mathrm{m}^\bullet$, i.e.\ a coupling of the form $(id,T^\bullet)_\#\xi^\bullet$, {where $S_\#\mu$ denotes the push-forward of the measure $\mu$ by the measurable map $S$.}

Now comes the important point: Note that $T^\bullet(x)-x$ is a stationary displacement field and any stationary displacement field induces a stationary matching. Moreover, if $\xi=\sum_{x\in\Zd}\delta_x$ the cost of a matching $m^\bullet=(id,T^\bullet)_\#\xi^\bullet$ is given by $\E[|T(0)-0|^p]$, the $L^p$ norm of the displacement field. 

Let us denote here by $\mathsf{P}^{\Zd}$ the stationary version of the integer lattice $\Zd$, i.e. $\mathsf{P}^{\Zd} := \mathrm{Law}(\sum_{x\in\Zd}\delta_{x+U})$, with a random variable $U$ uniformly distributed in $[0,1]^\d$. It is equivalent to find:
\begin{itemize}
   \item A $c_p$-optimal equivariant matching between $\mathsf{P}^{\Zd}$ and $\mathsf{P}$.
   \item A stationary displacement field  $(\p_x)_{x\in\Zd}$ with minimal $L^p$ norm $\E[|\p_0|^{{p}} ]$ such that the law of the resulting perturbed lattice $\PSL$ coincides with $\mathsf{P}$.
\end{itemize}

With this language, the main results of Theorem \ref{theo:main} and Theorem \ref{theo:classI} can be rephrased as follows:

\begin{theorem*}
{Let $\mathsf{P}$ be the law of a stationary point process $\bX$ on $\Rd$. We have:
\begin{enumerate}
\item If $\d=1,2$ and if $\Wass_\d(\mathsf{P}^{\Zd},\mathsf{P})$ is finite, then $\mathsf{P}$ is hyperuniform.
\item If $\d = 1, 2$, for all $\epsilon > 0$, there are examples with $\Wass_{\d-\epsilon}(\mathsf{P}^{\Zd},\mathsf{P})$ finite and yet $\mathsf{P}$ is not hyperuniform (in fact the number variance in finite balls is not even finite).
\item If $\d \geq 3$, for all $\epsilon > 0$, there are examples with $\Wass_{\infty}(\mathsf{P}^{\Zd},\mathsf{P}) \leq \epsilon$ and yet $\mathsf{P}$ is not hyperuniform.
\item If $\d=1$ and if $\Wass_2(\mathsf{P}^{\Z},\mathsf{P})$ is finite, then $\mathsf{P}$ is class-I hyperuniform. If $\d = 2$, for all $\epsilon > 0$, there are examples with $\Wass_{\infty}(\mathsf{P}^{\Z^2},\mathsf{P}) \leq \epsilon$ and yet $\mathsf{P}$ is not class-I hyperuniform.
\end{enumerate}}
\end{theorem*}

Since $\Wass_p$ is a geodesic metric, it follows that in dimensions $\d=1,2$ the set of point processes with finite $\Wass_p$ distance to $\mathsf{P}^{\Zd}$ is geodesically closed, and thus the geodesic interpolation between two hyperuniform point processes with finite $\Wass_\d$ distance to $\mathsf{P}^{\Zd}$ stays hyperuniform (resp. class-I hyperuniform for $\d=1$ and $\Wass_2$).

\begin{remark}\label{rem:lattice finite distance}
One can more generally define the Wasserstein distance between Radon measures (not necessarily purely atomic ones). In that sense, a stationary lattice is always at finite $\Wass_p$ distance of the Lebesgue measure (seen as a deterministic Radon measure) for all $p$ (in fact for $p = + \infty$). It is thus equivalent to be “at finite distance of a stationary lattice” or “at finite distance of Lebesgue”.
\end{remark}

\subsection{Open questions and connection with recent results}
\label{sec:open}

\subsubsection*{Perturbations of general hyperuniform processes}
Proving or disproving hyperuniformity of stationary perturbations of a general hyperuniform process (not only a stationary lattice) remains open. In fact, as we discuss below in Section \ref{sec:PertLatt}, even the correct notion of “perturbations” is not completely clear in general: for example, do we allow perturbations to depend on the realization of the underlying process or not? In \cite{gabrielli2004point}, the case of perturbations that are \emph{independent} from the process is considered.

For perturbations that can be correlated with the original point process, one can make the following observation: if $\bX$ is hyperuniform \emph{and has finite $2$-Wasserstein distance to a lattice}, then any point process $\bX'$ with finite $2$-Wasserstein distance to $\bX$ is hyperuniform again, by the triangle inequality and our main result. In other words, any $L^2$ perturbation of $\bX$ (independent of $\bX$ or not) will yield an hyperuniform process. Note that this does not completely answer the question because, as observed in \cite{HuesLeb}, not every hyperuniform process has finite $2$-Wasserstein distance to a lattice.

\subsubsection*{Converse statements in dimension $2$: the results of \cite{LRY,BDG,HuesLeb}}
Using the language introduced just above, our main result in the present paper is that in dimension~$2$, finite $2$-Wasserstein distance to a (stationary) lattice implies hyperuniformity. {Several recent papers are dealing (among other things) with the converse statement.}

In \cite{LRY}, the converse statement (hyperuniformity $\implies$ finite $2$-Wasserstein distance) was shown to be true \emph{under a certain integrability condition on the reduced pair correlation function} of the process which, they argue, proves that “\emph{most hyperuniform planar point processes are $L^2$-perturbed lattices}”.

{In \cite{BDG}, finite $2$-Wasserstein distance is obtained under a quantitative assumption on the decay of the rescaled number variance which is just slightly stronger than hyperuniformity itself.}

In a companion \cite{HuesLeb} to the present paper, the integrability condition of \cite{LRY} is observed to be closely related to some kind of “finiteness of the Coulomb energy” and it is shown that indeed:
\begin{equation*}
\text{Finite Coulomb energy } \implies \text{ Finite $2$-Wasserstein distance to a lattice } \implies \text{ Hyperuniformity},
\end{equation*}
(the second implication being the one from the present paper). However, simple examples are given in \cite{HuesLeb} of stationary point processes which are hyperuniform, but not at finite $2$-Wasserstein distance to a lattice - thus there is no perfect equivalence between those two notions. It is also proven that under a condition which is slightly stronger than hyperuniformity but slightly weaker than the one of \cite{BDG}, one obtains finite $2$-Wasserstein distance to a lattice. Since this condition covers by far the vast majority of known examples, it completely agrees with the statement of \cite{LRY} cited above. 

{The decay condition of \cite{HuesLeb}, called “$\star$-hyperuniformity”, is shown there to be equivalent to finite Coulomb energy. In the present paper, as a corollary of the proof of Theorem \ref{theo:main}, we obtain quantitative bounds on the decay of the rescaled number variance. In particular, for all $\epsilon>0$ (see Corollary \ref{cor:starHU}):
\begin{equation*}
\text{ Finite $2+\epsilon$-Wasserstein distance to a lattice } \implies \star-\text{hyperuniformity} \iff \text{ Finite Coulomb energy}.
\end{equation*}
}

\paragraph{Sufficient condition for class-I hyperuniformity}
In dimension $2$ we could not produce examples of perturbed lattices which are a) hyperuniform but not class-I hyperuniform \emph{and} b) such that the covariance structure of the perturbations satisfy a summability condition of the type: 
\begin{equation}
\label{covariancesummable}
\sum_{x \in \La} \left| \Cov(\p_x, \p_0) \right| < + \infty. 
\end{equation}
It is well-known that \eqref{covariancesummable} implies that the spectral density of the perturbations not only exists \emph{but is continuous}, which eliminates certain “bad” behaviors, {and could be enough to ensure that the resulting perturbed lattice is class-I hyperuniform.}

To see why, recall that the basic plan to break hyperuniformity (see e.g. \cite[eq. (34)]{kim2018effect}) is to choose perturbations $(\p_x)_{x \in \Zd}$ whose spectral density $\rho$ satisfies $\rho(\omega) \sim C |\omega|^{-2}$ as $\omega \to 0$. Indeed, when doing the computations (see e.g. \eqref{def:Cr} below), the behavior near $0$ of the spectral density of the perturbations turns out to play a major role for the asymptotic decay of the rescaled number variance, and $\rho(\omega) \sim C |\omega|^{-2}$ is in some sense the “critical” case separating hyperuniform and non-hyperuniform behaviors. Having (for the perturbations) $\rho(\omega) \sim C |\omega|^{-2+\delta}$ with  $\delta > 0$ produces a perturbed lattice which is hyperuniform, but for $\delta$ small it ought to be close to the “non-hyperuniform” case and in particular “not very hyperuniform” (type-I).

In dimension $2$, it is not possible to have $\rho(\omega) \sim C |\omega|^{-2}$ near $0$ because $\rho$ would fail to be integrable, but one can certainly produce Gaussian perturbation fields with $\rho(\omega) \sim C |\omega|^{-2+\delta}$ for any $\delta > 0$, and the resulting perturbed lattice will presumably \emph{not} be class-I hyperuniform if $\delta \leq 1$.  {However, if the spectral density $\rho$ is known to be continuous near $0$, such divergences are of course forbidden (one needs $\delta \geq 2$).} We are thus left with the following question: in dimension $2$, if $\E[|\p_0|^2] < + \infty$ and \eqref{covariancesummable} holds, is $\PSL$ necessarily class-I hyperuniform?

\subsection{Plan of the paper}
In Section \ref{sec:Prelim}, we introduce our notation, and our general setting. 

Section \ref{sec:d2p>2} is devoted to the proof of the main result of the paper: in dimension $\d=2$, a finite second moment of perturbations ensures hyperuniformity. 

The analogous positive result in dimension $\d=1$ is simpler to prove and is provided in Appendix \ref{sec:dim1}, along with the proof of the first item of Theorem \ref{theo:classI}. 

The proofs of other results, namely the first, third, and fourth items in Theorem \ref{theo:main}, the second item in Theorem \ref{theo:classI}, and Theorem \ref{theo:slow}, are of a different nature, involving more probabilistic constructions. These proofs are presented in Section \ref{sec:Counter}.

\section{Preliminaries}
\label{sec:Prelim}
\subsection{Some notation}
\label{Sec:Notations}
We write $A \preceq B$, or equivalently $A = \O(B)$, when there exists a constant $C$ depending only on the dimension $\d$, the lattice $\La$ and the law of the perturbations $(\p_x)_{x \in \La}$ such that $|A| \leq C |B|$. We write $f(x)=o(g(x))$ as $x\to a$ if $\lim_{x\to a}\frac{f(x)}{g(x)}=0$. If the value $a$ is clear from context, we do not mention it.

{In the proof of Proposition \ref{prop:L2D2}, which is our main result, we will work with a large parameter $r$ and auxiliary parameters depending on $r$. In that context, we sometimes write $A \ll B$ to signify that $\frac{A}{B} \to 0$ as $r \to \infty$ ($A$, $B$ being two quantities depending explicitly or implicitly on $r$).}

We use $|\cdot|$ to denote the cardinality of a finite set or the Lebesgue measure of a Borel set in $\Rd$. We write $\1_{\cdots}$ or $\1 \{\cdots\}$ for indicator functions.
We sometimes write $\left\langle \mu,f \right\rangle = \int f d\mu$ for the integral of a function $f$ against a (signed) measure $\mu$. We denote the convolution of two functions $f$ and $g$ by $f\ast g$.

\subsubsection*{Fourier transforms}
We choose the following convention for the Fourier transform: if $f$ is in $L^1(\Rd)$ we let:
\begin{equation*}
\hat{f}(s) := \int_{\Rd} e^{- i s \cdot t} f(t) \dd t.
\end{equation*}
We also define the characteristic function of a random variable $X$ as $\varphi_X(s):=\E[e^{-i s \cdot X}]$. With this convention we have $\hat{\hat{f}} = (2 \pi)^d f(- \cdot)$. We also have $\widehat{f \ast g} = \hat{f} \hat{g}$ but $\widehat{fg} = (2 \pi)^{-d} \hat{f} \ast \hat{g}$.

\subsection{Generalities on lattices}
\label{sec:Lattices}
\newcommand{\LaD}{\La^{\star}}

\subsubsection*{Some definitions}
We recall some definitions about lattices, which can be found in most textbooks.

The lattices considered here are subgroups of $\Rd$ generated by linear combinations \emph{with integer coefficients} of the vectors  of a certain basis of $\Rd$ (we only consider “full-rank” lattices). The basic example is the so-called integer lattice $\Zd$ (or “square lattice” for $\d =2$) which is simply the subgroup of all points in $\Rd$ with integer coefficients, or equivalently the set of all integer linear combinations of the vectors of the canonical basis of $\Rd$. Any other lattice can be obtained as an affine transformation of $\Zd$.
\begin{remark}
If we could show that any lattice $\La$ is itself a bounded stationary perturbation of $\Zd$, it would be enough to treat the case $\La = \Z^2$ for, say, the positive result of Theorem \ref{theo:main} in dimension $2$. This is indeed true, as they have both finite $\Wass_2$ distance to Lebesgue (see Remark \ref{rem:lattice finite distance}). However, we did not see an “easy” explicit perturbation that one could write down \emph{in a stationary fashion} to connect $\Zd$ in order to a general lattice, and the cost of working with general lattices is not too high, so we chose this option.
\end{remark}

To a lattice $\La$ and a basis $(\lambda_1, \dots, \lambda_\d)$, we associate a “fundamental domain” of $\La$ which is the set:
\begin{equation*}
\DD := \left\lbrace \sum_{i=1}^\dd t_i \lambda_i, \ (t_1, \dots, t_\d) \in [0,1)^\d \right\rbrace.
\end{equation*}
The fundamental domain is a compact parallelotope, whose volume does not depend on the choice of the basis and is called the “covolume” of $\La$.  \textit{In this paper, we will always assume the covolume to be $1$.} This is of course only for convenience and in no way a loss of generality. An interesting property is that every vector $t$ of $\Rd$ can be uniquely written as: $t = x + u$, where $x \in \La$ and $u \in \DD$. In other words, translating a fundamental domain by all the vectors in $\La$ makes a perfect tiling of $\Rd$. 

The dual lattice $\LaD$ of $\La$ is defined as: $\LaD := \left\lbrace k \in \Rd, \text{ such that } k \cdot x \in \Z \text{ for all } x \in \La \right\rbrace$. The only important fact for us is that $\LaD$ is indeed a lattice, and that it appears in Poisson's summation formula, which we recall next.

\subsubsection*{Poisson's summation formula}
We refer to \cite[Chap. VII]{stein1971introduction}. If $f : \Rd \to \mathbb{C}$ is in $L^1$ (so that $\hat{f}$ is well-defined) and if moreover there exists $\delta > 0$ and a constant $C > 0$ such that: $|f(x)| \leq \frac{C}{|x|^{\d + \delta}}, \quad |\hat{f}(x)| \leq \frac{C}{|x|^{\d + \delta}} \text{ as } |x| \to \infty$, then Poisson's summation formula holds, namely we have: 
\begin{equation}
\label{Poisson}
\sum_{x \in \La} f(x) = \sum_{k \in \LaD} \hat{f}(2\pi k).
\end{equation}
 Note that this obviously includes the case of $f$ in Schwartz's space (i.e.\ when $f$ is smooth and $f$ and all its derivatives decay faster than every polynomial at infinity).

\subsection{Generalities on point processes}
\label{sec:pointprocesses}
\newcommand{\RSFM}{\alpha_{\mathrm{red}}^{(2)}} 
\newcommand{\RSFC}{\gamma_{\mathrm{red}}^{(2)}} 
\newcommand{\jr}{j_r}
\newcommand{\hjr}{\widehat{\jr}}
\newcommand{\leb}{\lambda_d}
\newcommand{\BL}{\B_\ell}
\newcommand{\Bm}{\B_m}
\subsubsection*{Configuration of points}

{A deterministic configuration of points in $\Rd$ is defined as a locally finite integer-valued measure on $\Rd$, i.e. a measure of the form $\mathbf{x}=\sum_{i}\delta_{x_i}$ where $(x_i)_i$ is a finite or infinite sequence of points in $\Rd$ without accumulation points. Note that multiple points are allowed at the same location. The state space of all point configurations is endowed with the $\sigma$-field generated by the counting functions $\mathbf{x}\mapsto\mathbf{x}(B)$ for all Borel sets $B$ in $\Rd$. A random configuration of points, or “point process”, is a random variable on this state space. We refer to \cite[Sec. 2]{last2017lectures} for details about multiple points, labelling, etc.
}

\subsubsection*{Stationarity.}
If $\bX$ is a point configuration and $u$ is a vector in $\Rd$, we denote by $\bX + u$ the “shift by u” i.e.\ the effect of translating all the points of $\bX$ by $u$. A point process is said to be \emph{stationary} when its distribution is invariant with respect to all the translations i.e.\ when $\bX + u$ and $\bX$ have the same distribution for all $u \in \Rd$. 

When $\bX$ is stationary, we define its intensity as the average number of points of $\bX$ per unit volume. Here for convenience we always work with processes of intensity $1$.

\subsubsection*{Second order measures}
\label{sec:second_order}
When $\bX$ is a random point configuration, we say that $\bX$ is \emph{locally square integrable} if $\E[|\bX \cap \BR|^2]$ is finite for all $r > 0$. If $\bX$ is locally square integrable, we can define its \emph{reduced second factorial moment measure}, denoted by $\RSFM$. This measure plays a central role in the analysis of hyperuniformity for perturbed lattices because:
\begin{enumerate}
   \item The measure $\RSFM$ of a perturbed lattice has a tractable expression, see \eqref{alpha2PSL}.
   \item The number variance is conveniently expressed in terms of $\RSFM$, see \eqref{eq:NV}.
\end{enumerate}
We refer to \cite[Sections 4.4. \& 8.2]{last2017lectures} for details. 

\begin{definition}[Second factorial moment measure]
\label{def:MomentMeasure}
If $\bX$ is a point process, we define its \textit{second factorial moment measure} as the Borel measure $\alpha^{(2)}$ on $\R^\d \times \R^\d$ such that for all Borel sets $B_1, B_2$:
\begin{equation*}
\alpha^{(2)}(B_1 \times B_2):=\E\sum_{x_1,x_2 \in \bX}^{\neq} \I \left\lbrace x_1 \in B_1, x_2 \in B_2\right\rbrace.
\end{equation*}
If $\bX$ is locally square integrable, then $\alpha^{(2)}$ is well-defined, and if moreover $\bX$ is stationary, then we can define its \textit{reduced second factorial moment measure} $\alpha^{(2)}_{red}$ (sometimes known as \textit{pair correlation measure}), which is given, for $B$ a Borel set, by:
\begin{equation}
\label{eq_2ndRFMM}
\alpha_{red}^{(2)}(B)= \E \sum_{x,y \in \bX}^{\neq} \I\{x \in [0,1)^\d, y-x \in B\}.
\end{equation}
Note that in \eqref{eq_2ndRFMM} we can replace $[0,1)^\d$ by any Borel set of volume $1$ (see e.g. the proof \cite[Prop.~8.7]{last2017lectures}).
\end{definition}

\subsection{Basic properties of perturbed lattices}
\label{sec:PertLatt}
\subsubsection*{Discussion of the notion of “perturbations”}
Let $\La$ be a lattice and $\DD$ a fundamental domain of $\La$. Recall that we define the (unperturbed) stationary lattice as $\SL := \{x + U, x \in \La\}$ where $U$ is uniform in $\DD$. One could think of at least two ways to define “perturbations” of $\SL$:
\begin{enumerate}
   \item Let $\p$ be a random field $\La \mapsto \Rd$ - the value of $\p_x$ for $x \in \La$ corresponds to the “perturbation” applied to $x$.  Let $U$ be uniform in $\DD$ and independent of $\p$. We define the perturbed stationary lattice $\SL$ as the random collection of points given by $\PSL := \{ (x + \p(x)) + U, \ x \in \La \}$.
   \item Let $\bp$ be a random vector field on $\Rd$, thought of as a “displacement field” - a point placed at $x \in \Rd$ gets displaced by $\bp(x)$. Let $U$ be uniform in $\DD$ and independent of $\bp$. We define the perturbed stationary lattice $\SL$ as the random collection of points given by $\PSL' := \{ (x+U) + \bp(x+U), \ x \in \La \}$.
\end{enumerate}
It is straightforward to check that, when the perturbations themselves are stationary, those two notions coincide.
\begin{claim}
\label{claim:egalitenotions}
With $\PSL$ and $\PSL'$ as defined above:
\begin{enumerate}
   \item Assume that $\p$ is $\La$-invariant. For $t \in \Rd$, write $t = x + U$ with $x \in \La$ and $U$ in $\DD$, and set $\bp := \p(t- U)$. Then $\bp$ is a $\Rd$-invariant displacement field and $\PSL' = \PSL$.
   \item Assume that $\bp$ is $\Rd$-invariant and set $\p := \bp_{|\La}$. Then $\p$ is a $\La$-invariant perturbation field and $\PSL = \PSL'$. 
\end{enumerate}
\end{claim}
In the introduction, we chose the first option, but for our constructions in Section \ref{sec:Counter} it will be convenient to use the second option.

If one wanted to perturb hyperuniform point processes which are not lattices, the “displacement field” approach seems more tractable as it does not rely on an explicit labelling of the points. In that case however, the assumption of $\bp$ being independent of $\SL$ could be seen as restrictive (it is not very much so for lattices).
\begin{remark}
\label{rem:centeredness}
The properties of $\PSL$ are of course invariant under a global shift, thus we can and will assume that the perturbations (which are identically distributed) are centered, i.e.\ $\E[\p_0]=0$.
\end{remark}

\subsubsection*{Basic properties}
\begin{lemma}
\label{lem:basicppties}
If the perturbations are $\La$-invariant as assumed in Section \ref{sec:intro}, then:
\begin{enumerate}
   \item $\PSL$ is almost surely locally finite.
   \item $\PSL$ is a stationary point process.
\end{enumerate}
\end{lemma}
\begin{proof} For the first item, take $r \geq 10$ and bound $\E[|\PSL \cap \BR|]$ by:
\begin{equation*}
\E[|\PSL \cap \BR|] \leq \O(r^\d) + \sum_{x \in \La, |x| \geq 100r} \P\left[ \{x + \p_x + U \in \BR\} \right]. 
\end{equation*}
Let $C$ be an upper bound on the diameter of the fundamental domain $\DD$. We see that 
\begin{equation*}
\P\left[ \{x + \p_x + U \in \BR\} \right] \leq \P\left[ \{\p_x \in \B_{r + C}(-x) \} \right] = \P\left[ \{\p_0 \in \B_{r + C}(-x) \} \right],
\end{equation*}
using the fact that the perturbations are identically distributed. Finally, it remains to observe that:
\begin{equation*}
\sum_{x \in \La, |x| \geq 100r} \P\left[ \{\p_0 \in \B_{r + C}(-x) \} \right] \preceq r^\d
\end{equation*}
is finite, thus $\E[|\PSL \cap \BR|]$ is finite, hence $\PSL$ is almost surely locally finite.

For the second item, take $t \in \Rd$ and decompose $U + t$ as $x_\star + v$ with $x_\star \in \La$ and $v \in \DD$ (see Section \ref{sec:Lattices}). If $U$ is uniformly distributed in $\DD$, then so is $v$. We then have, using the $\La$-invariance of $(\p_x)_{x \in \La}$ the following identities in law:
\begin{equation*}
\{x + U + \p_x + t \}_{x \in \La} = \{x +  \p_x + x_\star + v \}_{x \in \La} = \{ (x + x_\star) +  \p_{x + x_\star} + v \}_{x \in \La} = \{ x  +  \p_{x} + v \}_{x \in \La}. 
\end{equation*}
\end{proof}

\subsubsection*{First finiteness result}
\begin{lemma}
\label{lem:FiniteVar}
If $\E[|\p_0|^\d]$ is finite, then $\PSL$ is locally square integrable in the sense of Section \ref{sec:second_order}.
\end{lemma}
\begin{proof}[Proof of Lemma \ref{lem:FiniteVar}]
We write, with $U$ uniformly distributed over the fundamental domain of $\La$:
\begin{multline*}
\E[|\PSL \cap \BR|^2] = \E \left(\sum_{x \in \La} \I\{ x + \p_x  + U \in \BR \} \right)^2  = \E [ \sum_{x, x' \in \La} \I\{ x + \p_x + U \in \BR \}  \times \I\{ x' + \p_{x'} + U \in \BR \} ]
\\ \leq \sum_{x, x' \in \La} \min\left( \P\left[ \{ x + \p_x +U \in \BR \} \right], \P\left[ \{ x' + \p_{x'} + U \in \BR \} \right] \right).
\end{multline*}
Let $C$ be an upper bound on the diameter of the fundamental domain $\DD$. Since the indices $x, x'$ play symmetrical roles we can write:
\begin{multline*}
\sum_{x, x' \in \La} \min\left( \P\left[ \{ x + \p_x  {+ U} \in \BR \} \right], \P\left[ \{ x' + \p_{x'} + U \in \BR \} \right] \right) \leq 2 \sum_{x, x' \in \La, |x'| \leq |x|} \P\left[ \{ x + \p_x + U\in \BR \} \right] \\
\preceq \sum_{x \in \La} |x|^\d \P\left[ \{ \p_x \in \B_{r+C}(-x) \} \right].
\end{multline*}
Using the symmetry of $\La$ under $x \mapsto -x$ and the fact that the random variables $(\p_x)_{x \in \La}$ are identically distributed, we also have:
\begin{equation*}
\sum_{x \in \La} |x|^\d \P\left[ \{ \p_x \in \B_{r+C}(-x) \} \right] = \sum_{x \in \La} |x|^\d \P\left[ \{ \p_0 \in \B_{r+C}(x) \} \right].
\end{equation*}
The contribution to the sum coming from $\{x \in \La, |x| \leq 2r\}$ is $\O(r^{\d})\times \O(r^{\d})$ and the rest of the sum can be compared to an integral (here, without loss of generality, we assume $r$ large enough so that $|x|^\d \leq 10 |x-r|^\d$ for $|x| \geq 2r$):
\begin{equation*}
\sum_{x \in \La, |x| \geq 2r} |x|^\d \P\left[ \{ \p_0 \in \B_{r + C}(x) \} \right] \preceq \int_{t \in \Rd, |t| \geq r} |t|^\d \P\left[ \{ \p_0 \in \B_{r+C}(t) \} \right] \dd t,
\end{equation*}
which itself, using Fubini's theorem, can be compared to an expression involving the law $P_0$ of $\p_0$, namely:
\begin{equation*}
\int_{t \in \Rd, |t| \geq r} |t|^\d \P\left[ \{ \p_0 \in \B_{r+C}(t) \} \right] \preceq r^\d \int_{t \in \Rd} |t|^\d \dd P_0(t),
\end{equation*}
and the last integral coincides with $\E[|\p_0|^\d]$, which is finite.
\end{proof}
In particular the reduced second factorial moment measure $\RSFM$ of $\PSL$ is well-defined. We now give its expression.

\subsubsection*{The reduced covariance function $\RSFM$ of perturbed lattices}
\begin{proposition}
\label{prop:ReducedCovPL}
For all bounded Borel subsets $B$ of $\Rd$, we have:
   \begin{equation}
   \label{alpha2PSL}
   \RSFM(B) = \E \sum_{x \in \La \setminus \{0\}} \I{\{x+\p_x-\p_0 \in B\}}.
   \end{equation}
\end{proposition}
\begin{proof}[Proof of Proposition \ref{prop:ReducedCovPL}]
We start from the definition \eqref{eq_2ndRFMM} of $\RSFM$, using the fundamental domain $\DD$ of $\La$ instead of $[0,1)^\d$ (this is valid because $\DD$ has volume $1$):
\begin{multline*}
\RSFM(B) = \E \sum^{\neq}_{x \in \PSL, y \in \PSL} \I\left\lbrace x \in \DD, y-x \in B \right\rbrace \\
= \E \sum^{\neq}_{x \in \La, y \in \La} \I\left\lbrace x + \p_x + U\in \DD, y + \p_y -x - \p_x \in B \right\rbrace \\
= \E \sum^{\neq}_{x \in \La, y \in \La} \I\left\lbrace \p_x +U \in \DD - x, y + \p_y -x - \p_x \in B \right\rbrace \\
= \E \sum_{x \in \La} \sum_{y \in \La \setminus \{x\}} \I\left\lbrace \p_0 +U \in \DD - x, y + \p_{y-x} -x - \p_0 \in B \right\rbrace \\
= \E \sum_{x \in \La} \sum_{y' \in \La \setminus \{0\}} \I\left\lbrace \p_0 +U \in \DD - x, y' + \p_{y'} - \p_0 \in B \right\rbrace \\
= \E \sum_{y' \in \La \setminus \{0\}} \I\left\lbrace \p_0 +U \in \cup_{x \in \La} \{ \DD - x \}, y' + \p_{y'} - \p_0 \in B \right\rbrace
\\
= \E \sum_{y' \in \La \setminus \{0\}} \I{\{y'+\p_{y'}-\p_0 \in B\}}.
\end{multline*}
In the first line we use the definition of $\PSL$ {(the symbol $\neq$ there refers to points in the perturbed lattice that are obtained as perturbations of two different lattice points (they could exceptionally end up at the same spot after perturbations)}, the second line is elementary, the third line uses the $\La$-invariance of the perturbations, the fourth line uses the fact that $\La$ is a group, the fifth line uses the fact {that the translates $\{ \DD - x \}$ of the fundamental domain are disjoint and the sixth line makes use of the fact that they cover $\Rd$} (see Section \ref{sec:Lattices}).
\end{proof}


\newcommand{\Rem}{\mathsf{Rem}}
\section{Positive result in dimension 2}
\label{sec:d2p>2}
This section is devoted to the proof of our main result in dimension $2$, namely:
\begin{proposition}
\label{prop:L2D2}
If the perturbations have a finite second moment, then $\PSL$ is hyperuniform.
\end{proposition}

\subsubsection*{Preliminaries}
\newcommand{\halpharedPL}{\widehat{\alpha}^{(2)}_{\mathrm{red}, \SL}}
\newcommand{\halpharedPSL}{\widehat{\alpha}^{(2)}_{\mathrm{red}, \PSL}}
\newcommand{\hchir}{\widehat{\chi_r}}
\newcommand{\chil}{\chi_\ell}
\newcommand{\hchil}{\hat{\chi}_\ell}
\newcommand{\chir}{\chi_r}
\label{subsecPrelim}
Without loss of generality, we assume that the perturbations are centered, see Remark \ref{rem:centeredness}. Since the perturbations have a finite second moment, by Proposition~\ref{prop:ReducedCovPL}, the measure $\RSFM$ of $\PSL$ exists. We denote by $\halpharedPSL$ its Fourier transform, and by $\halpharedPL$ the corresponding object for the stationary lattice.

In the sequel, we fix a smooth, compactly supported function $\chi : \R^2 \to [0,1]$ such that
\begin{equation*}
\chi \equiv 1 \text{ on } \B_1, \quad \chi \equiv 0 \text{ outside } \B_2,
\end{equation*}
and for $r > 0$ we define $\chi_r : \R^2 \to \R$ as
\begin{equation}
\label{eq:defchir}
\chi_r := s \mapsto r^2 \chi\left( r s \right).
\end{equation}
We then focus on showing that
\begin{equation}
\label{chihalpha}
\int_{\R^2} \chi_r(s) \left(\dd \halpharedPSL(s) - \dd \halpharedPL(s)\right) = o(1) \text{ as $r \to \infty$.}
\end{equation}
Indeed, we have, by our definition of $\chi$ and $\chir$:
\begin{multline*}
\int_{|s| \leq r^{-1}} \dd \halpharedPSL(s) \leq \frac{1}{r^2} \int_{\R^2} \chi_r(s) \dd \halpharedPSL(s) \\ = \frac{1}{r^2} \int_{\R^2} \chi_r(s) \dd \halpharedPL(s) + \left(\frac{1}{r^2} \int_{\R^2} \chi_r(s) \left(\dd \halpharedPSL(s) - \dd \halpharedPL(s)\right)  \right) \\
\leq \int_{|s| \leq 2 r^{-1}} \dd \halpharedPL(s) + \left(\frac{1}{r^2} \int_{\R^2} \chi_r(s) \left(\dd \halpharedPSL(s) - \dd \halpharedPL(s)\right)  \right).
\end{multline*}
In view of the spectral characterization of hyperuniformity recalled in Proposition \ref{prop:BH}, and since the non-perturbed stationary lattice is known to be hyperuniform of class I, we know that 
\begin{equation*}
\halpharedPL\left( \B_{2r^{-1}} \right) = \O(r^{-2-1}).
\end{equation*}
If \eqref{chihalpha} holds, we obtain:
\begin{equation*}
\halpharedPSL(\B_{r^{-1}}) = o(r^{-2}),
\end{equation*}
thus $\PSL$ is hyperuniform, and any quantitative upper bound of the form $\O(r^{-\gamma})$ in \eqref{chihalpha} translates into an estimate for the number variance as mentioned in Proposition \ref{prop:BH}. 

The starting point is to write the difference $\halpharedPSL - \halpharedPL$ using Proposition~\ref{prop:ReducedCovPL}, which yields:
\begin{equation*}
\int_{\R^2} \chi_r(s) \left(\dd \halpharedPSL(s) - \dd \halpharedPL(s)\right) = \sum_{x \in \La \setminus \{0\}} \E \left[\hchir(x+\p_x-\p_0) - \hchir(x) \right].
\end{equation*}

We will repeatedly use the fact that $\hchir(x) = \widehat{\chi}\left( \frac{x}{r} \right)$ and that since $\chi$ is smooth and compactly supported, its Fourier transform is Schwartz, in particular we have:
\begin{equation}
\label{deriveeschir}
|\hchir(x)| = \O\left( \frac{r^4}{|x|^4}\right),  \quad |\nabla \hchir(x)| = \O\left( \frac{r^3}{|x|^4}\right), \quad |\nabla^2 \hchir(x)| \leq \O\left( \frac{r^2}{|x|^4}\right) \text{ for } |x| \geq r.
\end{equation}

\subsubsection*{Truncation of the sum.}
\begin{claim}
Let $M_r$ be such that $M_r \gg r$ to be chosen later. We have:
\begin{equation}
\label{eq:Btruncate}
B = \sum_{x \in \La \setminus \{0\}, |x| \leq M_r} \E \left[\hchir(x+\p_x-\p_0)- \hchir(x) \right] + o_r(1),
\end{equation}
with an error term given by $\frac{r}{M_r} + \E\left[ |\p_0|^2  \1_{|\p_0| \geq \frac{1}{4} M_r} \right]$
\end{claim}
\begin{proof}
Fix $x \in \La \setminus \{0\}$ with $|x| \geq M_r$, and write $\p = \p_x - \p_0$ for simplicity. If $|\p| \leq \hal |x|$, we have
\begin{equation*}
\hchir(x+\p)- \hchir(x) = \nabla \hchir(x) \cdot \p + \O\left(|\p|^2 \frac{r^{2}}{|x|^{4}}\right),
\end{equation*}
using \eqref{deriveeschir} to bound the second derivative of $\hchir$. On the other hand, if $|\p| \geq \hal |x|$,  since $\hchir$ is bounded we have simply $\hchir(x+\p)- \hchir(x) = \O(1)$. Note that $|\p| \geq \hal |x|$ implies that $\max\left(|\p_0|, |\p_x|  \right) \geq \frac{1}{4} |x|$ and recall that $\p_0, \p_x$ are identically distributed, hence:
\begin{equation*}
\E\left[ |\hchir(x+\p)- \hchir(x)| \1_{|\p| \geq \hal |x|} \right] \preceq \P\left[ |\p_0| \geq \frac{1}{4} |x| \right].
\end{equation*}
Moreover, since the random variable $\p$ is centered, we can write:
\begin{equation*}
\E\left[ \nabla \hchir(x) \cdot \p \1_{|\p| \leq \hal |x|} \right] = - \E\left[ \nabla \hchir(x) \cdot \p \1_{|\p| \geq \hal |x|} \right] \preceq \frac{r^3}{|x|^4} \E\left[ |\p| \1_{|\p| \geq \hal |x|} \right] \preceq \frac{r^3}{|x|^5}  \E\left[ |\p|^2 \1_{|\p| \geq \hal |x|} \right] \preceq \frac{r^3}{|x|^5} \preceq \frac{r^2}{|x|^4}.
\end{equation*}
In summary, we obtain:
\begin{equation*}
\E \left[ \left|\hchir(x+\p_x-\p_0)- \hchir(x) \right|\right] \preceq \frac{r^{2}}{|x|^{4}} + \P \left[ |\p_0| \geq \frac{1}{4} |x| \right]
\end{equation*}
Summing over $x \in \La$ with $|x| \geq M_r$ yields:
\begin{multline*}
\sum_{x \in \La, |x| \geq M_r} \E \left[\hchir(x+\p_x-\p_0)- \hchir(x) \right] \preceq \sum_{x \in \La, |x| \geq M_r} \frac{r^{2}}{|x|^{4}}  + \sum_{x \in \La, |x| \geq M_r} \P \left[ |\p_0| \geq \frac{1}{4} |x| \right] \\
\preceq \frac{r}{M_r} + \E\left[ |\p_0|^2  \1_{|\p_0| \geq \frac{1}{4} M_r} \right],
\end{multline*}
which gives \eqref{eq:Btruncate}.
\end{proof}

Next, for each $x \in \La$, using the characteristic function $\varphi_{\p_x - \p_0}$ of $\p_x - \p_0$, we can write:
\begin{multline}
\label{totraceback}
\E \left[ \hchir(x+\p_x-\p_0)- \hchir(x) \right]  =  \E \left[ \int_{t \in \R^2} e^{-i t\cdot (x + \p_x - \p_0)} \chi_r(t) \dd t \right] - \int_{t \in \R^2} e^{-i t \cdot x} \chi_r(t) \dd t  \\ = \int_{t \in \R^2} e^{-i t\cdot x} \chi_r(t) \left( \E \left[ e^{-i t\cdot (\p_x - \p_0)}  \right] - 1 \right) \dd t 
= \int_{t \in \R^2} e^{-i t\cdot x} \chi_r(t) \left(\varphi_{\p_x - \p_0}(t) - 1 \right) \dd t.
\end{multline}
It is natural to introduce a small-$t$ expansion \eqref{eq:Taylor2} of the integrand, that we now present.

\subsubsection*{Taylor expansion of the characteristic function.}
We rely on a second order expansion of $\varphi_{\p_x - \p_0}(t)$ for $t$ small, which makes a crucial use of the existence of a second moment for the perturbations.
\begin{lemma}
\label{lem:Taylor}
There exists $g : \R^2 \to (0, + \infty)$ depending only on the distribution of $\p_0$ such that:
\begin{enumerate}
   \item $\lim_{t \to 0} g(t) = 0$, $g$ is bounded on $\R^2$. Moreover $g$ is radially symmetric and radially non-decreasing.
   \item For all $x \in \La$, we control the error term in the second-order Taylor's expansion of $\varphi_{\p_x - \p_0}$ by:
   \begin{equation}
   \label{eq:Taylor}
\left| \varphi_{\p_x - \p_0}(t) - \left(1 - \frac{1}{2} \E\left[ \big(t \cdot (\p_x - \p_0)\big)^2 \right] \right) \right| \leq g(t) \min\left(|t|^2, 1\right), \text{ for all } t \in \R^2.
   \end{equation}
\end{enumerate}
\end{lemma}
Using such a “small wave-number expansion” is a key part in the argument of \cite{gabrielli2004point}.
\begin{proof}[Proof of Lemma \ref{lem:Taylor}]
An elementary estimate found in \cite[(3.3.3)]{Durrett_2019} says that if $X$ is a random variable in $\Rd$, for all $n \geq 0$ we have the following control on the error term in the Taylor's expansion of the characteristic function of $X$ to order $n$: 
\begin{equation}
\label{durrett}
\left|\E[e^{-is \cdot X}] -  \sum_{k=0}^n \frac{(-i)^k}{k!} \E[(s \cdot X)^k] \right| \leq \E \left[ \min(|s \cdot X|^{n+1}, 2|s \cdot X|^n) \right] \text{ for all $s \in \Rd$.}
\end{equation}
For $x \in \La$, we can use \eqref{durrett} with $X = \p_x - \p_0$ and $n = 2$, which gives:
\begin{equation*}
\left|\varphi_{\p_x - \p_0}(t) - \left(1 - \frac{1}{2} \E\left[  \big(t \cdot (\p_x - \p_0)\big)^2 \right] \right) \right| \preceq |t|^2 \E \left[ \min\left(|t| |\p_x - \p_0|^3, |\p_x - \p_0|^2\right) \right].
\end{equation*}
Using elementary inequalities and the fact that $\p_x$ and $\p_0$ are identically distributed, we can write
\begin{equation}
\label{symDur}
|t|^2 \E \left[ \min(|t| |\p_x - \p_0|^3, |\p_x - \p_0|^2) \right] \preceq |t|^2 \E \left[ \min(|t| |\p_0|^3, |\p_0|^2) \right].
\end{equation}
The right-hand side of \eqref{symDur} is radially symmetric, radially non-decreasing, does not depend on $x$ and is $o(|t|^2)$ for $t\to0$ by dominated convergence (the fact that the convergence is indeed dominated relies on the fact that $\p_0$ has a finite second moment). This shows that we can indeed take $g(t) \to 0$ as $t \to 0$ in \eqref{eq:Taylor}. On the other hand, for large values of $t$, since $\varphi_{\p_x - \p_0}$ is bounded by $1$, the error in \eqref{eq:Taylor} is exactly of order $|t|^2$ and we can take $g(t)$ of order $1$. 

In general, we cannot control the “speed” for the decay of $g$. Note however that if we placed a stronger moment assumption on $\p_0$, e.g. by asking that $\E[|\p_0|^{2 + \delta}]$ be finite, then we could take an explicit form for $g(t)$ as $|t|^{\delta}$ for $|t|$ small.
\end{proof}

In the sequel, we may thus use the following expansion for $|t|$ small:
\begin{equation}
\label{eq:Taylor2}
\varphi_{\p_x - \p_0}(t) = 1 - \frac{1}{2} \E\left[  \big(t \cdot (\p_x - \p_0)\big)^2\right] + o(|t|^2),
\end{equation}
with an error term which may depend on $x$ but is bounded by $g(t)|t|^2$ uniformly in $x$.

\subsubsection*{Inserting the Taylor expansion into the truncated sum.}
Thanks to \eqref{totraceback} and to our Taylor expansion, we have for all $x$:
\begin{equation*}
\E \left[ \hchir(x+\p_x-\p_0)- \hchir(x) \right] = - \hal  \int_{t \in \R^2} e^{-i t\cdot x} \chi_r(t) \E\left[ \big(t \cdot (\p_x - \p_0)\big)^2 \right]  \dd t \\
+ \int_{t \in \R^2} e^{-i t\cdot x} \chi_r(t) o(|t|^2) \dd t.
\end{equation*}
We may thus write the sum in the right-hand side of \eqref{eq:Btruncate} as: 
\begin{equation*}
\sum_{x \in \La \setminus \{0\}, |x| \leq M_r} \E \left[\hchir(x+\p_x-\p_0)- \hchir(x) \right] = B_1 + B_2,
\end{equation*}
where $B_1$ is the “main term”:
\begin{equation}
\label{def:B_1}
B_1 : = - \hal \sum_{x \in \La \setminus \{0\}, |x| \leq M_r}   \int_{t \in \R^2} e^{-i t\cdot x} \chi_r(t)  \E\left[ \big(t \cdot (\p_x - \p_0)\big)^2 \right] \dd t,
\end{equation}
and $B_2$ is the error due to the small wave-number Taylor's expansion:
\begin{equation}
\label{def:B2}
 B_2 : = \sum_{x \in \La \setminus \{0\}, |x| \leq M_r} \int_{t \in \R^2} e^{-i t\cdot x} \chi_r(t) o(|t|^2) \dd t.
 \end{equation}
 We call $B_1$ the “main” term because it carries the most tractable information about the perturbations. However, our goal is to show that \emph{both terms $B_1, B_2$ are going to $0$} as $r \to \infty$. 

\subsection{The main term \texorpdfstring{$B_1$}{B1}}
\label{sec:main}
\newcommand{\T}{\mathbb{T}}
\newcommand{\gl}{g_\ell}
\newcommand{\hgl}{\widehat{\gl}}
\begin{lemma}
\label{lem:B1}
We have $B_1 = o_r(1)$, with an error term given by $\O\left(\frac{r}{M_r}\right) + C(r)$, where $C(r) \to 0$ in such a way that $\sum_{n} C(2^n) < + \infty$. 
\end{lemma}
\begin{proof}[Proof of Lemma \ref{lem:B1}] 
$B_1$ is written as a finite sum, but we would like to apply Poisson's formula, so we start by re-introducing the missing terms.

\paragraph*{Re-introducing some terms in the sum}
\newcommand{\psiab}{\psi^{ab}}
\newcommand{\psiabr}{\psi^{ab}_r}
\newcommand{\hpsiab}{\widehat{\psiab}}
\newcommand{\hpsiabr}{\widehat{\psiabr}}
For fixed $x \in \La$, write
\begin{equation*}
\int_{t \in \R^2} e^{-i t\cdot x} \chi_r(t) \big(t \cdot (\p_x - \p_0)\big)^2 \dd t = \sum_{a,b \in \{1,2\}} (\p_x -\p_0)_a (\p_x -\p_0)_b \int_{t \in \R^2} e^{-i t\cdot x} \chi_r(t) t_a t_b \dd t,
\end{equation*}
where $a, b$ denote the first or second coordinates in $\R^2$. Taking the expectation and using the finite second moment of the perturbations, we get:
\begin{equation*}
\int_{t \in \R^2} e^{-i t\cdot x} \chi_r(t) \E\left[ \big(t \cdot (\p_x - \p_0)\big)^2 \right] \dd t =  \sum_{a,b \in \{1,2\}} \alpha_{ab}(x) \int_{t \in \R^2} e^{-i t\cdot x} \chi_r(t) t_a t_b \dd t,
\end{equation*}
all the $\alpha_{ab}(x) := \E\left[   (\p_x -\p_0)_a (\p_x -\p_0)_b \right]$ being $\O(1)$ uniformly in $x$. Define $\psiabr$ as:
\begin{equation*}
\psiabr(t) := \chi_r(t) t_a t_b = r^2 \chi(rt) t_a t_b = \frac{1}{r^2} \times r^2 \psiab(rt) \text{ where } \psiab(s) := \chi(s) s_a s_b,
\end{equation*}
so that in particular we have $\hpsiabr(x) = \frac{1}{r^2} \hpsiab\left( \frac{x}{r} \right)$. We have obtained:
\begin{equation*}
\int_{t \in \R^2} e^{-i t\cdot x} \chi_r(t) \E\left[ \big(t \cdot (\p_x - \p_0)\big)^2 \right]\dd t = \sum_{a,b \in \{1,2\}} \alpha_{ab}(x) \frac{1}{r^2} \hpsiab\left( \frac{x}{r} \right).
\end{equation*}
The map $\psiab$ is compactly supported, hence its Fourier transform decays fast, for instance for $|x| \geq r$ we have $\hpsiab\left( \frac{x}{r} \right) \preceq \frac{r^{4}}{|x|^{4}}$. We get, for $|x| \geq M_r \gg r$:
\begin{equation*}
\int_{t \in \R^2} e^{-i t\cdot x} \chi_r(t) \E\left[ \big(t \cdot (\p_x - \p_0)\big)^2 \right] \dd t \preceq \frac{r^{2}}{|x|^{4}}.
\end{equation*}
Summing over $x \in \La$ with $|x| \geq M_r$ yields:
\begin{equation*}
\sum_{x \in \La, |x| \geq M_r} \int_{t \in \R^2} e^{-i t\cdot x} \chi_r(t) \E\left[ \big(t \cdot (\p_x - \p_0)\big)^2 \right] \dd t  \preceq \frac{r}{M_r}.
\end{equation*}
We may thus re-introduce the “missing” terms in the sum defining $B_1$ and write:
\begin{equation}
\label{B1nonmissing}
B_1 = - \hal \sum_{x \in \La}   \int_{t \in \R^2} e^{-i t\cdot x} \chi_r(t) \E\left[ \big(t \cdot (\p_x - \p_0)\big)^2 \right] \dd t + o_r(1),
\end{equation}
with an error term bounded by $\O\left(\frac{r}{M_r}\right)$. 

\paragraph*{The spectral measure(s)}
Expanding the inner product again, we write:
\begin{equation}
\label{B1dev}
 -\frac 12\sum_{x \in \La}  \int_{t \in \R^2} e^{-i t\cdot x} \chi_r(t) \E\left[ \big(t \cdot (\p_x - \p_0)\big)^2 \right] = -\frac 12 \sum_{a,b \in \{1,2\}} \sum_{x \in \La} \E\left[(\p_x - \p_{0})_a (\p_x - \p_{0})_b \right] \hpsiabr(x).
\end{equation}  
\newcommand{\DDs}{\DD^\star}
We introduce the \emph{autocorrelation functions} of the field $(\p_x)_{x \in \La}$ (recall that the perturbations are assumed to be centered):  for $x \in \La$ we let 
\begin{equation*}
G^1(x) := \E[\p_{x,1} \p_{0,1}]\quad \text{ and} \quad G^2(x) := \E[\p_{x,2} \p_{0,2}]
\end{equation*}
as well as 
\begin{equation*}
G^{1,2+}(x) := \E[(\p_{x,1}+\p_{x,2}) (\p_{0,1}+\p_{0,2})]\quad \text{ and} \quad G^{1,2-}(x) := \E[(\p_{x,1}-\p_{x,2})  (\p_{0,1}-\p_{0,2})].
\end{equation*}
Since $(\p_x)_{x \in \La}$ is translation-invariant, it is classical that:
\begin{itemize}
   \item These autocorrelation functions are bounded by their values at $0$.
   \item $G^1$, $G^2$ and $G^{1,2\pm}$ are positive-definite functions on $\La$. In particular, by Bochner's theorem (see \cite[Thm. (4.18)]{folland2016course}), there exist four “spectral measures” $\rho^1,\rho^2, \rho^{1,2\pm} $ which are non-negative finite measures on the Pontryagin dual $\Omega$ of the locally compact abelian group $\La$, such that for all $ x \in \La$ and $m=1,2$
   \begin{equation}
   \label{GasTFrho}
   G^m(x) = \int_{\Omega} e^{- 2i\pi x \cdot \omega} \dd \rho^m(\omega)\quad \text{and} \quad  G^{1,2\pm}(x) = \int_{\Omega} e^{- 2i\pi x \cdot \omega} \dd \rho^{1,2\pm}(\omega).
   \end{equation}
   {Note that the Pontryagin dual of $\La$ is different from, although related to, the \emph{dual lattice} $\LaD$.}
When $\La = \Z^2$, its Pontryagin dual is the two-dimensional torus $\T^2$ which we can identify with $[-\hal, \hal)^2$. In general, we identify $\Omega$ with a copy of \emph{the fundamental domain $\DDs$ of} $\LaD$ centered at the origin (this is convenient for the organization of our computation below, but in no way necessary).
\end{itemize}
\newcommand{\hpsi}{\widehat{\psi}}
\newcommand{\hpsiUU}{\widehat{\psi^{11}_r}}
\newcommand{\hpsiDD}{\widehat{\psi^{22}_r}}
\newcommand{\hpsiUD}{\widehat{\psi^{12}_r}}
\newcommand{\psiUU}{{\psi^{11}_r}}
\newcommand{\psiDD}{{\psi^{22}_r}}
\newcommand{\psiUD}{{\psi^{12}_r}}

We can express the expectations in \eqref{B1dev} using these autocorrelation functions by noting first that for $a \in \{1,2\}$ we have $\E\left[ (\p_x - \p_0)_a^2 \right] = 2 (G^a(0) - G^a(x))$, and that moreover:
\begin{equation*}
 \E\left[(\p_{x} - \p_{0})_1 (\p_{x} - \p_{0})_2 \right]
= \frac 12  \left( G^{1,2+}(0)-G^{1,2+}(x) - G^{1,2-}(0)+G^{1,2-}(x)\right).
\end{equation*}
We may now return to \eqref{B1dev}, use our notation, and write 
\begin{equation*}
-\frac 12 \sum_{a,b \in \{1,2\}} \sum_{x \in \La} \E\left[(\p_x - \p_{0})_a (\p_x - \p_{0})_b \right] \hpsiabr(x) = B_{11} - B_{12},
\end{equation*}
where
\begin{equation*}
B_{11} := \sum_{x \in \La} G^1(x)\hpsiUU(x) +G^2(x) \hpsiDD(x)+\frac{1}{2}(G^{1,2+}(x)-G^{1,2-}(x)) \hpsiUD(x) 
\end{equation*}
and 
\begin{equation*}
B_{12} := \sum_{x \in \La} G^1(0)\hpsiUU(x) +G^2(0)\hpsiDD(x) +\frac{1}{2}(G^{1,2-}(0)-G^{1,2+}(0))\hpsiUD(x).
\end{equation*}

\paragraph*{Controlling $B_{11}$}
Using the spectral measures $\rho^1,\rho^2$ and $\rho^{1,2\pm}$ as in \eqref{GasTFrho}, we can re-write $B_{11}$ as:

\begin{multline*}
B_{11} = \int_{\Omega} \sum_{x \in \La} e^{- 2i\pi x \cdot \omega} \hpsiUU(x) \dd \rho^1(\omega) + \int_{\Omega} \sum_{x \in \La} e^{- 2i\pi x \cdot \omega} \hpsiDD(x) \dd \rho^2(\omega) 
\\
+ \hal \int_{\Omega}  \sum_{x \in \La} e^{- 2i\pi x \cdot \omega} \hpsiUD(x)  \dd (\rho^{1,2+}-\rho^{1,2-})(\omega).
\end{multline*}
Applying Poisson's formula \eqref{Poisson} for all fixed $\omega \in \Omega$, we obtain:
\begin{multline*}
\label{B_11Poi}
B_{11} = (2 \pi)^2 \sum_{k \in \LaD} \Big( \int_{\Omega} \psiUU(2\pi k + 2\pi \omega) \dd \rho^1(\omega) + \int_{\Omega}  \psiDD(2\pi k + 2\pi \omega) \dd \rho^2(\omega) \\ + \frac 12  \int_{\Omega}  \psiUD(2\pi k + 2\pi \omega) \dd (\rho^{1,2+}-\rho^{1,2-})(\omega)\Big).
\end{multline*}
Note that all the $\psi_r^{ab}$ are Schwartz functions, which justifies our interversions of sums over $x \in \La$ and integrals against $ \rho^1$, $\rho^2$, $\rho^{1,2+}$ and $\rho^{1,2-}$ as well as the use of Poisson's formula. We now bound each $\psi_r^{ab}(\cdot)$ by $\chi_r(\cdot) |\cdot |^2$ and get the rough bound:
\begin{equation}
|B_{11}| \preceq  \sum_{k \in \LaD} \int_{\Omega}  \chi_r(2\pi k + 2\pi \omega) |k + \omega|^2 \dd \rho(\omega),
\end{equation}
where $\rho$ is the finite positive measure given by $\rho := \rho^1+\rho^2+\frac 12\rho^{1,2+}+\frac 12\rho^{1,2-}$. Next, we distinguish between the lattice points $k = 0$ and $k \in \LaD \setminus \{0\}$. 
\subparagraph*{For $k \neq 0$.}
{We chose $\Omega$ as a parallelotope (because it is the fundamental domain of some lattice) containing $0$.}
For all $k \in \LaD \setminus \{0\}$, the function $\omega \mapsto k + \omega$ stays away from $0$ on $\Omega$ in the sense that:
\begin{equation*}
\min_{k \in \LaD \setminus \{0\}} \min_{\omega \in \Omega} |k + \omega| \geq c := \dist(0, \partial \Omega) > 0.
\end{equation*}
{(In the case $\La = \LaD = \Z^2$, we are simply saying that if $k \in \Z^2 \setminus \{0\}$ and $u \in [\hal, \hal)^2$, then $|k + u|$ is bounded below by a positive constant).} Thus, since $\chi_r$ is supported on $\B_{2r^{-1}}$, if $r$ is large enough then we have $\chi_r(2\pi k + 2\pi \omega) = 0$ for all $k \in \LaD \setminus \{0\}$ and all $\omega \in \Omega$, and in turn:
\begin{equation*}
\sum_{k \in \LaD \setminus\{0\}} \int_{\Omega} \chi_r(2\pi k + 2\pi \omega) |k + \omega|^2 \dd \rho(\omega) = 0.
\end{equation*}

\newcommand{\Aa}{\mathtt{A}}

\subparagraph*{For $k = 0$.} Denote by $C(r)$ the quantity:
\begin{equation}
\label{def:Cr}
C(r) := \int_{\Omega} \chi_r(2\pi \omega) |\omega|^2 \dd \rho(\omega).
\end{equation}
Let us now justify that $C(r)$ tends to $0$ as $r \to \infty$. The family of functions $\omega \mapsto \chi_r(2\pi \omega) |\omega|^2$ converges pointwise to $0$ on $\Omega$ as $r \to \infty$ because $\chi_r$ converges pointwise to $0$ on $\R^2 \setminus \{0\}$. Moreover, the quantity $\chi_r(2\pi \omega) |\omega|^2 = r^2 |\omega|^2 \chi(2\pi r \omega)$ is bounded uniformly in $r, \omega$ because $\chi$ is compactly supported. Thus by the dominated convergence theorem (this is where we use the fact that the measure $\rho$ is finite), we have:
\begin{equation}
\label{eq:TCD}
\lim_{r \to 0} \int_{\Omega} \chi_r(2\pi \omega) |\omega|^2  \dd \rho(\omega) = 0.
\end{equation}
This is a point where the speed of decay is more difficult to quantify. For example, imposing additional moment conditions on the perturbations does not help. We can however make the following observation:
\begin{claim}
\label{preHUstar}
We have $\lim_{r \to \infty} C(r) = 0$ in such a way that:
\begin{equation*}
\sum_{n \geq 0} C(2^n) < + \infty.
\end{equation*}
\end{claim}
\begin{proof}[Proof of Claim \ref{preHUstar}]
We have by definition of $\chi_r$ as $r^2$ times a cutoff function supported on $B_{2r^{-1}}$:
\begin{equation}
\label{useasymp}
\int_{\Omega} \chi_r(2\pi \omega) |\omega|^2  \dd \rho(\omega) \preceq r^2 \int_{|\omega| \leq \frac{1}{\pi r}} |\omega|^2 \dd \rho(\omega).
\end{equation}
For $j \geq 0$, let $\Aa(j)$ be the annulus $\{2^{-j-1} \leq |\omega| < 2^{-j}\}$. In view of \eqref{useasymp}, we can compare $C(r)$ to the following sum:
\begin{equation*}
C(r) \preceq \frac{1}{r^2} \sum_{j=0}^{\log_2 r -1} 2^{-2j} \rho(\Aa(j)).
\end{equation*}
In particular, if we take $r = 2^n$ and sum over $n$, we obtain, since $\rho$ is a finite measure:
\begin{equation*}
\sum_{n \geq 0} C(2^n) \preceq \sum_{n \geq 0} 2^{-2n} \sum_{j=0}^{n-1} 2^{2j} \rho(\Aa(j))
= \sum_{j = 0}^{+\infty}  2^{2j} \rho(\Aa(j)) \sum_{n \geq j+1} 2^{-2n}  \preceq \sum_{j=0}^{+\infty} \rho(\Aa(j)) < + \infty.
\end{equation*}
\end{proof}

\paragraph*{Handling $B_{12}$}
We are left with the term $B_{12}$. Using Poisson's formula, we get:
\begin{equation*}
B_{12} = (2 \pi)^2  \sum_{k \in \LaD} G^1(0) \psiUU(2 \pi k) + G^2(0) \psiDD(2 \pi k) + \frac 12 (G^{1,2-}(0)-G^{1,2+}(0)) \psiUD(2\pi k),
\end{equation*}
which is $0$ for $r$ large enough. Indeed, the $\psi_r^{ab}$ vanish at $0$ and are supported on $\B_{2r^{-1}}$.
\end{proof}

\paragraph*{Conclusion}
We have obtained that $\lim_{r \to 0} B_1 = 0$. More precisely, for $r$ large enough, $B_1$ reduces to the sum of two error terms: the one from \eqref{B1nonmissing}, which is $\O\left( \frac{r}{M_r} \right)$, and the one corresponding to $C(r)$ as in \eqref{def:Cr}, which is such that $\sum_{n \geq 0} C(2^n) < + \infty$ by Claim \ref{preHUstar}.

\subsection{The error term \texorpdfstring{$B_2$}{B2}}
Recall that:
\begin{equation*}
B_2 : = \sum_{x \in \La \setminus \{0\}, |x| \leq M_r} \int_{t \in \R^2} e^{-i t\cdot x} \chi_r(t) o(|t|^2) \dd t.
\end{equation*}

\begin{lemma}
\label{lem:B2}
We have $B_2 \preceq \left(\frac{M_r}{r}\right)^2 g(2r^{-1})$, with $g$ as in Lemma \ref{lem:Taylor}.
\end{lemma}
\begin{proof}[Proof of Lemma \ref{lem:B2}]
We use a rough triangular inequality and control $B_2$ by:
\begin{equation*}
B_2 = \sum_{x \in \La \setminus \{0\}, |x| \leq M_r} \int_{t \in \R^2} e^{-i t\cdot x} \chi_r(t) o(|t|^2) \dd t \preceq M_r^2 \int_{t \in \R^2} \chi_r(t) |t|^2 g(t) \dd t, 
\end{equation*}
where $g$ is as in Lemma \ref{lem:Taylor}. We conclude using the fact that $\chi_r$ is supported on $\B_{2r^{-1}}$, has mass $1$, and that $g$ is radially non-decreasing.
\end{proof}

\subsection{Conclusion}
Gathering our estimates and returning to \eqref{chihalpha}, we find that:
\begin{equation*}
\int_{\R^2} \chi_r(s) \left(\dd \halpharedPSL(s) - \dd \halpharedPL(s)\right) \preceq \frac{r}{M_r}  + \E\left[ |\p_0|^2  \1_{|\p_0| \geq \frac{1}{4} M_r} \right] + C(r) + \left(\frac{M_r}{r}\right)^2 g(2r^{-1}),
\end{equation*}
where:
\begin{itemize}
\item $M_r \gg r \gg 1$
\item $\E\left[ |\p_0|^2  \1_{|\p_0| \geq \frac{1}{4} M_r} \right] \to 0$ as $r \to 0$
\item $C(r) \to 0$ in such a way that $\sum_{n} C(2^n) < \infty$
\item $g(h) \to 0$ as $h \to 0$
\end{itemize}
To prove Proposition~\ref{prop:L2D2}, it remains to choose $M_r$ such that 
\begin{equation*}
\lim_{r \to \infty} \frac{r}{M_r} = 0, \quad \lim_{r \to \infty} \left(\frac{M_r}{r}\right)^2 g(2r^{-1}) = 0,
\end{equation*}
which can be done by setting e.g.
\begin{equation*}
M_r := r \left( g(2r^{-1}) \right)^{-1/4}.
\end{equation*}

\paragraph{Speed of decay.} We obtain $\halpharedPSL(\B_r)  = o(r^{-2})$ (spectral hyperuniformity). In general, there is no upper bound on the speed of decay - arbitrarily slow decays are possible, see Theorem \ref{theo:slow}, but we if we assume that $|\p_0|$ has a moment of order $2 + \alpha$ with $\alpha > 0$, then we can say the following:
\begin{itemize}
   \item The function $g$ from Lemma \ref{lem:Taylor} can be taken as a $\O(|t|^{\alpha})$.
   \item We can then take $M_r$ as $r^{1+ \alpha'}$ for some $\alpha' > 0$
   \item The error term $\E\left[ |\p_0|^2 \1_{|\p_0| \geq M_r} \right]$ becomes $\O(r^{-\alpha''})$ for some $\alpha'' > 0$.
\end{itemize}
In this case, we obtain
$$
\halpharedPSL(\B_r) \preceq r^{-2 - \alpha'''} + r^{-2} C(r), \text{ with } \sum_{n} C(2^n) < + \infty,
$$
which implies the following “spectral condition” (see \cite{sodin2023random})
\begin{equation*}
\int_{0 \leq |\omega| \leq 1} \frac{1}{|\omega|^2} \dd \halpharedPSL(\omega) < \infty,
\end{equation*}
which in turn is known to imply (\cite[Lemma 4.4]{HuesLeb}) that:
\begin{equation}
\label{thenHUs}
\sum_{n \geq 0} \sigma(2^n) < + \infty.
\end{equation}
Hyperuniform point processes whose rescaled number variance satisfies \eqref{thenHUs} are called “$\star$-hyperuniform” in \cite{HuesLeb}, where it is argued that they form an interesting sub-class (in fact, the overhelming majority). Moreover, $\star$-hyperuniformity is equivalent to finite (regularized) Coulomb energy (see \cite[Theorem 4]{HuesLeb} for further details).

\begin{corollary}\label{cor:starHU}
{Let $\d=2$ and assume that $\E[|\p_0|^{2+\alpha}]<\infty$ for some $\alpha>0$. Then, the resulting perturbed lattice $\PSL$ is $\star$-hyperuniform. In particular, it has finite (regularized) Coulomb energy.}
\end{corollary}

\section{Constructions}
\label{sec:Counter}
\newcommand{\LakL}{\Lambda^{(k)}_N}
\newcommand{\LaL}{\Lambda_N}
\newcommand{\LaUL}{\Lambda^{(1)}_N}
\newcommand{\LaZL}{\Lambda^{(0)}_N}
\newcommand{\pL}{\alpha_N}
\newcommand{\tp}{\tilde{\p}_N}
\newcommand{\CkN}{C^{(k)}_N}
\newcommand{\CZN}{C^{(0)}_N}

For $N \geq 10$ we let $\LaL=\LaL^{(0)}$ be the (half-open) cube $\LaL:= [-\hal N, \hal N)^\d$, and for all $k \in N\Zd$ we let $\LakL := k + \LaZL$, which is a cube of sidelength $N$ centered at $k$.

\subsection{Infinite number variance without \texorpdfstring{$\d$-th moment}{d-th moment}}
\label{sec:Counter2}

We prove here the third item of Theorem \ref{theo:main}.
\begin{proposition}
\label{prop:CEd=2}
Let $\d \geq 1$ and let $X$ be a non-negative random variable with $\E[X^\d] = + \infty$. Then there exists a $\Zd$-invariant perturbation field $(\p_x)_{x \in \Zd}$ such that:
\begin{enumerate}
   \item $|\p_0| \leq X$ almost surely.
   \item The resulting perturbed lattice $\PSL$ has infinite number variance in finite balls.
\end{enumerate}
\end{proposition}

\begin{proof}[Proof of Proposition \ref{prop:CEd=2}]
{Let $(\CkN)_{k \in N \Zd}$ be independent random variables such that $\CkN$ is uniformly distributed in $\LakL$ for all $k \in N \Zd$.}
 We first define a “displacement field” (as in Section \ref{sec:PertLatt}) $\tp : \Rd \mapsto \Rd$ by setting, for all $k \in N\Zd$: 
 \begin{equation*}
\tp(x) := \CkN - x, \text{ for all } x \in \LakL,
 \end{equation*}
hence applying $\tp$ corresponds to “sending all the points of the cube $\LakL$ to {$\CkN$}. This displacement field is clearly $N\Zd$-stationary, we now turn it into a $\Rd$-stationary field by choosing a random vector $\tau$ uniformly distributed in $\LaL$ and independent of the $(\CkN)_{k \in N \Zd}$, and setting $\bp_N := \tp(\cdot + \tau)$. Note that by construction, \emph{the size of all displacements is bounded by $\sqrt{\d} N$}. Finally, we set for a {$[0,1]^d$-uniform random variable $U$ (independent of everything else):}
 \begin{equation*}
\PSL'_N := \{x + U + \bp_N(x+U), \ x \in \Zd \}.
 \end{equation*} 
Next, we count the number of points of the corresponding perturbed lattice $\PSL'_N$ falling in the unit ball. 

With probability greater than $\frac{|\B_1|}{(2N)^\d}$ the random variable {$\CZN - \tau$ (which is the difference of two independent random vectors both uniformly distributed in $\LaL^{(0)}$)} is in $\B_1$, and under this event if $y$ is a point in $\LaZL - \tau$, then $\bp_N(y) = \tp(y+\tau) = \CZN - (y + \tau)$, consequently if $x$ is a lattice point in $\Zd$ such that $x+U \in \LaZL - \tau$, then we have: 
\begin{equation*}
x + U + \bp_N(x+U) = x + U + \tp(x+U + \tau) = x + U + {\CZN} - (x + U + \tau) = {\CZN - \tau} \in \B_1.
\end{equation*}
By a simple lattice point counting, and since we take $N \geq 10$, there exists a constant $c$ depending only on the dimension such that for all values of $\tau$ and $U$ there are at least $c N^\d$ points of $\Zd + U$ in the cube $\LaZL - \tau$. In conclusion, under the event $\{\CZN - \tau \in \B_1\}$, the number of points of $\PSL'_N$ in $\B_1$ is larger than $c N^\d$. The second moment of the number of perturbed points falling in $\B_1$ is thus bounded below by:
\begin{equation}
\label{LB_numberB1}
\E[|\PSL'_N \cap \B_1|^2] \geq (c N^\d)^2 \times \P(\{\CZN - \tau \in \B_1\}\}) \geq c' N^{\d}.
\end{equation}
We are now ready to construct our example.

\textit{{Infinite number variance without $\d$-th moment.}}
 If $X$ is a non-negative random variable, we let $\mathcal{N} := \frac{1}{\sqrt{\d}} \floor{X}$ and we define our perturbed lattice as $\PSL := \PSL'_{\mathcal{N}}$, so that $\PSL$ is a mixture of the perturbed lattices $\PSL'_N$ defined above, the parameter $N$ being chosen at random according to $\frac{1}{\sqrt{\d}} \floor{X}$. 

On the one hand, since the perturbations in $\PSL'_N$ are bounded almost surely by $\sqrt{\d}N$ for all $N$, the size of a single perturbation $\p_0$ in $\PSL$ is bounded by $X$ almost surely. On the other hand, using \eqref{LB_numberB1} we see that:
\begin{equation}
\label{LB_numberB1_Full}
\E[|\PSL \cap \B_1|^2] \geq c' \sum_{N \geq 10} N^{\d} \times \P[\mathcal{N} = N], \quad \text{ with } \mathcal{N} = \frac{1}{\sqrt{\d}} \floor{X},
\end{equation}
and thus if $X$ fails to have a finite moment of order $\d$ then certainly the variance of $|\PSL \cap \B_1|$ is infinite.
\end{proof}

\subsection{Arbitrarily slow decay of the rescaled number variance: proof of Theorem \ref{theo:slow}}
\label{sec:arbitraryslowd=2}
The processes that we just used for the proof of Proposition \ref{prop:CEd=2} can also be used to construct hyperuniform processes with arbitrarily slow decay of the rescaled number variance.
\def\tsigma{\tilde{\sigma}}
\begin{proposition}
\label{prop:slowdecay}
Let $\tsigma : (1, + \infty) \to (0, + \infty)$ be a non-increasing function such that $\lim_{r \to \infty} \tsigma(r) = 0$. There exists a $\Zd$-invariant perturbation field $(\p_x)_{x \in \Zd}$ such that:
\begin{enumerate}
   \item The resulting perturbed lattice $\PSL$ is hyperuniform (i.e.\ $\lim_{r \to \infty} \sigma(r) = 0$).
   \item There exist $c > 0, M \geq 1$ (depending on $\d$ and $\sigma$) such that $\PSL$ satisfies, for $r \geq 1$:
\begin{equation}
\label{slowdecay}
\sigma(r) := \frac{\Var[|\PSL \cap \BR|]}{|\BR|} \geq c \tsigma(Mr).
\end{equation}
\end{enumerate}
\end{proposition}
\begin{proof}[Proof of Proposition \ref{prop:slowdecay}]
The letters $c, c', c''$... denote small constants that may change from line to line.

We use the same family of processes $\{\PSL'_N\}_{N \geq 1}$ as in the proof of Proposition~\ref{prop:CEd=2}. With a similar argument as the one leading to \eqref{LB_numberB1}, we find that if $r \geq 1$ and $N \geq 10r$, we have:
\begin{equation*}
\E[|\PSL'_N \cap \BR|^2] - |\BR|^2 \geq c N^{2\d} \times \P(\{\CZN - \tau \in \BR\}\}) - |\BR|^2 \geq c' N^{\d} r^\d - |\B_1|^2 r^{2\d},
\end{equation*}
because $\P(\{{\CZN - \tau} \in \B_r\})$ is larger than $c \frac{r^\d}{N^\d}$ {and under this event, all the $N^\d$ points of $\LaL - \tau$ will “jump” inside $\BR$}. Taking $M$ large enough (depending only on $\d$) we thus see that for $N \geq Mr$, we have:
\begin{equation}
\label{varBr}
\Var[|\PSL'_N \cap \BR|] \geq c'' N^\d r^\d.
\end{equation}
Now let us consider a mixture of the $\{\PSL'_N\}_{N \geq 1}$ of the form $\PSL := \sum_{N \geq 1} \alpha_N \PSL'_N$, with coefficients $\alpha_N$ chosen as:
\begin{equation*}
\alpha_N := \frac{1}{S} \frac{-\tsigma '_N}{N^\d}, \quad \tsigma '_N := \tsigma(N+1) - \tsigma(N), \quad S := \sum_{N \geq 1} \frac{-\tsigma '_N}{N^\d}.
\end{equation*}
{(Note that thanks to a summation by parts, the series $\sum_{N \geq 1} \frac{-\tsigma'_N}{N^\d}$ has the same nature as $\sum_{N \geq 1} \frac{\tsigma_N}{N^{\d +1}}$, which converges for all $\d \geq 1$ because $(\tsigma_N)_N$ is bounded, and thus $S$ is well-defined.)} Using \eqref{varBr} and computing a telescopic sum, we obtain:
\begin{equation}
\label{lowsharp1d}
\frac{\Var[|\PSL \cap \BR|]}{|\BR|}  \geq c \sum_{N \geq Mr} \alpha_N N^\d = \frac{c}{S} \sum_{N \geq M r} - \tsigma'_N \geq c \tsigma(Mr),
\end{equation}
for some positive constant $c$, which gives \eqref{slowdecay}. It remains to prove that $\PSL$ is hyperuniform. 

For $\d = 1, 2$ we can compute, using again a telescopic sum, the $\d$-th moment of the perturbations:
\begin{equation*}
\sum_{N \geq 1} \alpha_N \times N^\d = \frac{1}{S} \times \sum_{N \geq 1} -\tsigma'_N < + \infty,
\end{equation*}
and thus (by our main result) $\PSL$ is indeed hyperuniform. {This condition is not sufficient in dimension $\d \geq 3$, but we can instead directly find an upper bound on the number variance. Let us return to the “building block” $\PSL'_N$ and observe that:
\begin{itemize}
   \item If $N \geq r$, then $\Var[|\PSL'_N \cap \BR|] = \O(r^{\d} N^\d)$. Indeed, with probability $\O\left(\frac{r^\d}{N^\d}\right)$ the ball $\BR$ receives $N^\d$ points, and it loses all of its $\O(r^\d)$ points with probability $1 - \O\left(\frac{r^\d}{N^\d}\right)$.
   \item If $r \geq N$ we have:
\begin{equation*}
\Var[|\PSL'_N \cap \BR|] \preceq \frac{r^{\d-1}}{N^{\d-1}} \times N^{2\d} = \O\left(r^{\d-1} N^{\d+1}\right).
\end{equation*}
Indeed, there are $\O\left(\frac{r^{\d-1}}{N^{\d-1}}\right)$ hypercubes of sidelength $N$ intersecting the hypersphere $\partial \BR$, and those are the only ones contributing (independently) to the variance, each one with a contribution $\O(N^{2\d})$. 
\end{itemize}
Summing over $N$, we obtain:
\begin{equation*}
\frac{\Var[|\PSL \cap \BR|]}{|\BR|} \preceq  \frac{1}{r} \sum_{N \leq r} \alpha_N N^{\d+1} +  \sum_{N \geq r} \alpha_N N^\d = \frac{1}{r} \sum_{N \leq r} - \tsigma'_N N + \sum_{N \geq r} - \tsigma'_N 
\preceq \frac{1}{r} \sum_{N \leq r} \tsigma_N + \tsigma(r) \to 0,
\end{equation*}
so the process that we constructed is indeed hyperuniform.
}
\end{proof}

\begin{remark}
\label{rem:sharpd1}
For $\d = 1$, using the same construction with $\alpha_N$ of order $N^{-3 + \frac{\epsilon}{2}}$ for some $\epsilon \in (0,1)$ gives a perturbation such that:
\begin{itemize}
\item $\sum_{N \geq 1} \alpha_N N^{2-\epsilon} < + \infty$, hence $\mathbb{E}[|\p_0|^{2-\epsilon}]$ is finite.
\item In view of the first inequality in \eqref{lowsharp1d}, 
\begin{equation*}
\frac{\Var[|\PSL \cap \BR|]}{|\BR|}  \geq c \sum_{N \geq Mr} \alpha_N N = c' \sum_{N \geq Mr} N^{-2 + \frac{\epsilon}{2}} \geq c'' r^{-1 + \frac{\epsilon}{2}}
\end{equation*}
which means that the resulting perturbed lattice is \emph{not} type-I hyperuniform.
\end{itemize}
This confirms our “sharpness” claim in the first item of Theorem \ref{theo:classI}.
\end{remark}

\subsection{Arbitrarily small perturbations with large number variance}
We now prove the fourth item of Theorem \ref{theo:main} as well as the second item of Theorem \ref{theo:classI}.
\label{sec:Counter3+}
\begin{proposition}
\label{prop:bounded}
Let $\epsilon  >0$. 
\begin{itemize}
   \item If $\d \geq 3$, there are perturbations of $\Zd$ that are almost surely bounded by~$\epsilon$ and such that 
   \begin{equation*}
   \lim_{r \to + \infty} \frac{\Var\left[|\PSL \cap \BR|\right]}{|\BR|} = + \infty
   \end{equation*}
   \item If $\d = 2$, there are perturbations of $\Z^2$ that are almost surely bounded by $\epsilon$ and such that $\PSL$ is not hyperuniform of class I. 
\end{itemize}
\end{proposition}
\begin{proof}
Let $N \geq 10$ and let $(\CkN)_{k \in N \Zd}$ be independent random variables such that $\CkN$ is uniformly distributed in $\LakL$ for all $k \in N \Zd$. We define a first “displacement field” $\tp : \Rd \mapsto \Rd$ by setting:
\begin{equation*}
\tp(x) := \epsilon \frac{x-\CkN}{|x-\CkN|} \text{ for all $k \in N\Zd$ and all $x$ in $\LakL$. }
\end{equation*}
Applying $\tp$ corresponds to choosing a center $\CkN$ uniformly at random in each cube $\LakL$ and then push every point in $\LakL$ by an amount $\epsilon$ away from $\CkN$. This displacement field is clearly $N\Zd$-stationary, and we can turn it into a $\Rd$-stationary field by choosing a random vector $\tau$ uniformly distributed in $\LaL$ and independent of $\tp$ and setting $\bp_N := \tp(\cdot + \tau)$. Perturbations are bounded by $\epsilon$. Finally, we set for a {$[0,1]^\d$-uniform random variable $U$ (idependent of everything else):}  
\begin{equation*}
\PSL'_N := \{x + U + \bp_N(x+U), \ x \in \Zd \}.
\end{equation*} 
Next, let $r$ be such that ${N \geq r}$. We want to count the number of points of $\PSL'_N$ falling in $\B_{r}$, and we will show that a significant amount of the original lattice points are “leaving” the ball with positive probability, creating a {very} large {number} variance if $\d \geq 3$ {and a critical number variance if $\d =2$}. We rely on the following elementary computation.
\begin{claim}
\label{claim:elementarygeom}
Let $r >0 , x, u, \tau, C \in \Rd$ be such that:
\begin{equation*}
\frac{9}{10}r \leq |x+u| \leq \frac{11}{10}r, \quad |\tau - C| \leq \frac{1}{10}r.
\end{equation*}
Then we have:
\begin{equation}
\label{rayongrandit}
\left|x + u + \epsilon \frac{x + u + \tau - C}{|x + u + \tau - C|}\right| \geq |x+u| + \frac{\epsilon}{3}.
\end{equation}
\end{claim}
\begin{proof}[Proof of Claim \ref{claim:elementarygeom}]
We simply compute:
\begin{equation*}
\left|x + u + \epsilon \frac{x + u + \tau - C}{|x + u + \tau - C|}\right|^2 = |x + u|^2 + \epsilon^2 + 2 \epsilon (x+u) \cdot \frac{x + u + \tau - C}{|x + u + \tau - C|}.
\end{equation*}  
With our assumptions, we know that $\frac{8}{9} |x+u| \leq |x+u + \tau - C| \leq \frac{10}{9} |x+u|$, and thus:
\begin{multline*}
(x+u) \cdot \frac{x + u + \tau - C}{|x + u + \tau - C|} = |x+u|^2 \frac{1}{|x + u + \tau - C|} + (x+u) \cdot (\tau -C) \frac{1}{|x + u + \tau - C|} \\
\geq |x+u|^2 \times \frac{9}{10 |x+u|}  - |x+u| \times \frac{1}{10} r \times \frac{9}{8|x+u|}
 \geq \frac{9}{10} \times \frac{9}{10} r - \frac{9}{80} r \geq \hal r.
\end{multline*}
As a consequence, we see that (using again $|x+u| \leq \frac{11}{10}r$):
\begin{equation*}
\left|x + u + \epsilon \frac{x + u + \tau - C}{|x + u + \tau - C|}\right|^2  \geq |x + u|^2 + \epsilon r + \epsilon^2 \geq |x+u|^2 + 2 \frac{10}{2 \times 11} |x+u| \epsilon + \epsilon^2 \geq \left(|x+u| + \frac{\epsilon}{3} \right)^2, 
\end{equation*}
which proves the claim.
 \end{proof}
Since $\tau$ and $\CZN$ are uniformly distributed in $\LaL$ and independent, there exists $c > 0$ depending only on the dimension, such that:
\begin{equation}
\label{probatauCZNgood}
\P\left(  \left\lbrace |\tau| \leq \frac{r}{10} \right\rbrace \cap \left\lbrace \left|\tau - \CZN\right| \leq \frac{r}{10} \right\rbrace  \right) \geq c {\frac{r^{2\d}}{N^{2\d}}}.
\end{equation}
This event is interesting for us due to the following observation:
\begin{claim}
\label{claim:quiquestla}
On the event $\left\lbrace |\tau| \leq \frac{r}{10} \right\rbrace \cap \left\lbrace \left|\tau - \CZN\right| \leq \frac{r}{10} \right\rbrace$, for all $u \in [0,1]^\d$, if 
$\frac{9}{10}r \leq |x+u| \leq \frac{11}{10}r$ then $|x+u + \bp(x+u)| \geq |x+u| + \frac{\epsilon}{3}$.

In particular, if $x$ is such that $|x+u| \geq r - \frac{\epsilon}{4}$ then $x + u + \bp(x+u) \notin \BR$.
\end{claim}
\begin{proof}[Proof of Claim \ref{claim:quiquestla}]
 If $|x+u| \leq \frac{11}{10r}$ then $|x + u + \tau| \leq \frac{12}{10}r$, which means that $x+u+\tau \in \LaZL$ (because $N \geq 5r$) and thus, by construction of the displacement field:
\begin{equation*}
\bp(x+u) = \tp(x+u + \tau) = \epsilon \frac{x+u+\tau-\CZN}{|x+u+\tau-\CZN|}.
\end{equation*}
We may then apply Claim \ref{claim:elementarygeom} with $C = \CZN$, which yields $|x+u + \bp(x+u)| \geq |x+u| + \frac{\epsilon}{3}$. Next, let $x$ be such that $|x+u| \geq r - \frac{\epsilon}{4}$.
\begin{enumerate}
   \item If $|x+u| \geq \frac{11}{10}r$ then $x+u + \bp(x+u)$ cannot be in $\BR$ because perturbations are bounded by $\epsilon < \frac{1}{10}r$.
   \item If $|x+u| \leq \frac{11}{10}r$, then since by assumption $|x+u| \geq r - \frac{\epsilon}{4} \geq \frac{9}{10}r$ we get from above that:
   \begin{equation*}
   |x + u + \bp(x+u)| \geq |x+u| + \frac{\epsilon}{3} \geq r - \frac{\epsilon}{4} + \frac{\epsilon}{3} > r
   \end{equation*}
   and thus $x + u + \bp(x+u)$ is not in $\BR$.
\end{enumerate}
\end{proof}

From Claim \ref{claim:quiquestla} we deduce that on the event $\left\lbrace |\tau| \leq \frac{r}{10} \right\rbrace \cap \left\lbrace \left|\tau - \CZN\right| \leq \frac{r}{10} \right\rbrace$ we have:
\begin{equation}
\label{eq:condiU}
\left|\PSL'_N \cap \BR\right| \leq \left|\left\lbrace x \in \Zd, |x+U| \leq r - \frac{\epsilon}{4}\right\rbrace \right| \text{ a.s.}
\end{equation}
We are thus left with finding an upper bound for the number of integer points in the ball $\B_{r - \frac{\epsilon}{4}}(-U)$, which we expect from the continuous approximation to be of order $|\BR| - C r^{\d-1} \epsilon$, and thus significantly smaller than $|\BR|$. Although the quest for sharp bounds seems to still be an active area of research, “basic” estimates are sufficient for us and have thankfully been known for a long time - as the classical result of Landau \cite{landau1915analytischen}, which gives:
\begin{equation}
\label{eq:Landau}
\sup_{a \in \Rd} \left| \left| \left\lbrace x \in \Zd, x \in \BR(a) \right\rbrace  \right| - |\BR| \right| \leq C_{\d} r^{\frac{\d(\d-1)}{\d+1}} \ll r^{\d-1},
\end{equation}
so in particular it ensures that for all $u\in [0,1]^d$ we have:
\begin{equation}
\label{appliLandau}
\left|\left\lbrace x \in \Zd, |x+u| \leq r - \frac{\epsilon}{4}\right\rbrace \right| = |\B_{r-\frac{\epsilon}{4}}| + o(r^{\d-1}) \leq |\BR| - C_\d r^{\d-1} \epsilon + o(r^{\d-1}),
\end{equation}
{so in particular $|\BR| - \left|\PSL'_N \cap \BR\right| \geq C'_\d r^{\d-1} \epsilon$.}
From this {and our lower bound \eqref{probatauCZNgood} on the probability of} $\left\lbrace |\tau| \leq \frac{r}{10} \right\rbrace \cap \left\lbrace \left|\tau - \CZN\right| \leq \frac{r}{10} \right\rbrace$  we deduce: 
\begin{equation}
\label{eq:rNbons}
\text{For $r, N$ large enough {with $N \geq r$},} \quad \E\left[\left||\BR| - |\PSL'_N \cap \BR|\right|^2\right] \geq c_\d r^{2\d -2} {\times \frac{r^{2\d}}{N^{2\d}}},
\end{equation}
for some positive constant $c_\d$ depending on the dimension.  We are now ready to construct our examples.

\vspace{0.1cm}

\textit{Non-hyperuniformity for $\d \geq 3$.} If $\d \geq 3$, fix some integer-valued random variable $\mathcal{N}$ with $\P[\mathcal{N} = N]$ proportional to $\frac{1}{N^{1 + \delta}}$ for $N \geq 1$ (where $\delta \in (0,1)$ is arbitrary), and set $\PSL := \PSL'_{\mathcal{N}}$ i.e.\ $\PSL$ is a mixture of the perturbed lattices $\PSL'_N$ defined above, the parameter $N$ being chosen at random. For any $r$, a lower bound on the number variance in $\BR$ is given by:
\begin{equation}
\label{lowerboundPSLBRepsilon}
\E\left[\left||\PSL \cap \BR| - |\BR|\right|^2\right]  \geq \sum_{N \geq r}  \E\left[\left||\PSL'_N \cap \BR| - |\BR|\right|^2\right] \P[\mathcal{N} = N].
\end{equation}
From \eqref{eq:rNbons} we deduce that, for $r$ large enough (depending on $\d$):
\begin{equation*}
\frac{\Var[|\PSL \cap \BR|]}{|\BR|} \geq c_\d \frac{r^{2\d - 2}}{r^\d} \times \sum_{5r \geq N \geq r} \frac{1}{N^{1+\delta}},
\end{equation*}
which after a simple computation is seen to be greater than $r^{2 \d - 2 - \d - \delta}$, which is $\gg 1$ for all $\d \geq 3$.

\vspace{0.1cm}
\textit{Non-class-I hyperuniformity for $\d = 2$.}
The same construction for $\d = 2$ yields $\frac{\Var[|\PSL \cap \BR|]}{|\BR|} \geq r^{-\delta}$ where $\delta > 0$ can be arbitrary small, hence this process is not class I hyperuniform. 

{Note that one could go further and consider coefficients $\P[\mathcal{N} = N]$ proportional to $\frac{1}{N \log^{1+\delta}N}$, which would yield a decay of the rescaled number variance slower than any power law.}
\end{proof}

\appendix

\section{The case of dimension 1}
\subsection{Expressing the number variance in balls}
\subsubsection*{An auxiliary function}
\label{sec:jr}
For $r > 0$, a classical computation shows that the Fourier transform of the indicator function of a ball of radius $r$ is given by: $\widehat{\1_{\BR}} = \frac{\Jd(r |t|)}{|t|^{\frac{\d}{2}}} r^{\frac{\d}{2}}$, where $\Jd$ is a Bessel function of the first kind (whose definition is not needed here). As a consequence, we have $\frac{\widehat{\1_{\BR} \ast \1_{\BR}}}{|\BR|} = C_\d \frac{\Jd(r |t|)^2}{|t|^{\d}}$ for some constant $C_\d$. We define the function $\jr : \Rd \to \R$ by: $\jr : t \mapsto C'_\d \frac{\Jd(r |t|)^2}{|t|^{\d}}$ with the right constant $C'_\d$ so that its Fourier transform satisfies:
\begin{equation}
\label{hjr}
\hjr = \frac{\\1_{\BR} \ast \1_{\BR}}{|\BR|}.
\end{equation}
For all $r > 0$, we have $\hjr(0) = \int_{\Rd} j_r(t) \dd t = 1$ and 
\begin{equation*}
\jr(0) = \frac{|\BR|}{(2\pi)^d}
\end{equation*} 
The family of functions $\{j_r\}_{r \geq 1}$ is an approximation of unity: the mass stays constant equal to $1$ and the functions converge pointwise to $0$ on $\Rd \setminus \{0\}$, while the value at $0$ tends to $+ \infty$.

\subsubsection*{Expressing the number variance}
When $\bX$ is a random point configuration, the following is well-known (see e.g. \cite[Sec.~1.2]{coste2021order}):
\begin{lemma}
\label{lem:expressing_NV}
Assume $\bX$ is stationary, with intensity $1$ and locally square integrable. For all $r > 0$:
\begin{equation}
\label{eq:NV}
\frac{\Var\left[ |\bX \cap \BR| \right]}{|\BR|} = \left\langle \delta_0 + (\RSFM-1) , \frac{\I_{\BR} \ast \I_{\BR}}{|\BR|} \right\rangle = \left\langle \RSFM , \frac{\I_{\BR} \ast \I_{\BR}}{|\BR|} \right\rangle + 1 - |\BR|.
\end{equation}
\end{lemma}
In view of \eqref{hjr} one can re-write \eqref{eq:NV} as:
\begin{equation}
\label{eq:NVdivided}
\frac{\Var\left[ |\bX \cap \BR| \right]}{|\BR|} = \left\langle \RSFM, \hjr \right\rangle + 1 - |\BR|,
\end{equation}
and the hyperuniformity of $\bX$ is thus equivalent to having $\left\langle \RSFM, \hjr \right\rangle = |\BR| - 1 +  o(1)$, while class I hyperuniformity is equivalent to  $\left\langle \RSFM, \hjr \right\rangle  = |\BR| - 1 + \O(r^{-1})$ as $r \to \infty$.

\subsubsection*{Perturbed lattices in dimension 1}
\newcommand{\bd}{\bar{d}}

\label{sec:dim1}
\begin{proposition}
\label{prop:dimension1}
If $\E[|\p_0|]$ is finite, then $\PSL$ is hyperuniform. If $\E[|\p_0|^2]$ is finite, then $\PSL$ is class-I hyperuniform.
\end{proposition}
\begin{proof}
We start with introducing an auxiliary function.

\paragraph{Preliminary}
We record the following elementary lemma.
\begin{lemma}
\label{lem:remainder}
Let $f : \Rd \to \R$ be bounded and compactly supported in $\Bm$ for some $m \geq 1$. Assume that $M \geq 10 m$. We have:
\begin{equation}
\label{eq:remainder}
\E \sum_{x \in \La, |x| \geq M} f(x+\p_x-\p_0) \preceq \|f\|_{\infty} \times \E \left[ |\p_0|^\d \times \I \left\lbrace |\p_0| \geq \frac{1}{4} M\right\rbrace \right].
\end{equation}
\end{lemma}
Note that if $\E[|\p_0|^\d] < + \infty$, then the right-hand side of \eqref{eq:remainder} tends to $0$ as $M \to \infty$. 
\begin{proof}[Proof of Lemma \ref{lem:remainder}] Without loss of generality, we can assume that $0 \leq f \leq 1$. {Let $|x| \geq M$. Since $f$ is supported on $\Bm$, then for $f(x + \p_x - \p_0)$ to not be $0$ we need $|\p_x - \p_0| \geq |x| - m \geq \hal |x|$, so $\max(|\p_x|, |\p_0|) \geq \frac{1}{4}{|x|}$. We thus have, since $\p_x$ and $\p_0$ are identically distributed:
\begin{equation*}
\sum_{|x| \geq M} \E  f(x+\p_x-\p_0) \preceq \sum_{|x| \geq M} \P\left(|\p_0| \geq \frac{1}{4}|x| \right),
\end{equation*}
and by elementary lattice counting this discrete sum is comparable to the integral: 
\begin{equation*}
\int_{s \geq \frac{M}{4}} s^{\d-1} \P\left(|\p_0| \geq s \right) \preceq \E\left[ |\p_0|^\d \I \left\lbrace |\p_0| \geq \frac{M}{4} \right\rbrace\right].
\end{equation*}}
\end{proof}

\paragraph{Main computation}
We start by writing, using \eqref{eq:NVdivided} and the expression for $\RSFM$ found in \eqref{alpha2PSL}:
\begin{equation*}
\frac{\Var\left[ |\PSL \cap \BR| \right]}{|\BR|} =  \E \sum_{x \in \La \setminus \{0\}} \hjr(x+\p_x-\p_0) + 1 - |\BR| = A + B, 
\end{equation*}
where $A$ is a deterministic contribution coming from the lattice $\Z$ and $B$ is the “error term” due to the perturbations:
\begin{equation*}
A :=  \sum_{x \in \Z \setminus \{0\}} \hjr(x) + 1 - |\BR|, \quad B :=  \E \sum_{x \in \Z \setminus \{0\}} \hjr(x+\p_x-\p_0) - \hjr(x).
\end{equation*}
The term $A$ is $\O(r^{-1})$ because a stationary lattice is class-I hyperuniform, and we focus on $B$.

We use Lemma \ref{lem:remainder} and write:
\begin{equation}
\label{truncationlattice}
B = \E \sum_{x \in \Z \setminus \{0\}, |x| \leq 100 r} \hjr(x+\p_x-\p_0) - \hjr(x) + \O\left(\E\left[|\p_0| \1_{|\p_0| \geq 10 r}\right]\right).
\end{equation}
Note that if $\E[|\p_0|]$ is finite then the error term tends to $0$ as $r \to \infty$ and if $\E[|\p_0|^2]$ is finite, then by Markov's inequality it is actually $\O(r^{-1})$. 

The function $\hjr$ is a “tent” supported on $[-2r, 2r]$ and whose expression is given by $\hjr(t) = 1 - \frac{|t|}{2r}$ on that interval. In particular, it is continuous, $\O(r^{-1})$-Lipschitz and piecewise affine, with a lack of differentiability at $0$ and $\pm 2r$. For $x \in \Z$, we let $d := \dist\left(x, \{0, -2r, 2r\}\right)$. We introduce some lengthscale $\bd$ to be chosen later with $1 \leq \bd \ll r$. First, using the Lipschitz bound on $\hjr$ we have:
\begin{equation}
\label{prochebord}
\E \sum_{x \in \Z, d \leq \bd}  \hjr(x+\p_x-\p_0) - \hjr(x) \preceq \frac{\bd}{r} \times \E[|\p_0|].
\end{equation}
Next, if $x$ is such that $d \geq \bd$, we write:
\begin{equation*}
\E [\hjr(x+\p_x-\p_0) - \hjr(x)] = \E \left[ \left(\hjr(x+\p_x-\p_0) - \hjr(x) \right) \1_{|\p_x - \p_0| \leq d}\right] + \E \left[ \left(\hjr(x+\p_x-\p_0) - \hjr(x)\right) \1_{|\p_x - \p_0| \geq d}\right].
\end{equation*}
Since $\hjr$ is affine on $(x-d, x + d)$ {with slope $\O(r^{-1})$}, and $\p_x-\p_0$ is centered, we have:
\begin{equation*}
\E \left[ \left(\hjr(x+\p_x-\p_0) - \hjr(x) \right) \1_{|\p_x - \p_0| \leq d}\right] \leq \frac1{r}\E\left[ |\p_x - \p_0|\1_{|\p_x - \p_0| \geq d} \right]\preceq  \frac{1}{r} \E \left[ |\p_0| \1_{|\p_0| \geq \frac{d}{2}}\right].
\end{equation*}
It remains to see that, {since $\hjr$ is $\O(r^{-1})$-Lipschitz}:
\begin{equation}
\label{simpleA}
\E \left[  \left(\hjr(x+\p_x-\p_0) - \hjr(x)\right) \1_{|\p_x - \p_0| \geq d}\right] \preceq \frac{1}{r} \E \left[ |\p_x - \p_0| \1_{|\p_x - \p_0| \geq d}\right] \preceq \frac{1}{r} \E \left[ |\p_0| \1_{|\p_0| \geq \frac{d}{2}}\right] ,
\end{equation}
from which we can derive the following simple bound:
\begin{equation}
\label{simpleB}
\E \sum_{x \in \Z, d \geq \bd, |x| \leq 100 r} \hjr(x+\p_x-\p_0) - \hjr(x) \preceq \E \left[ |\p_0| \1_{|\p_0| \geq \frac{\bd}{2}} \right].
\end{equation}
If $\E[|\p_0|]$ is finite, the right-hand side goes to $0$ as soon as $\bd \to \infty$. On the other hand, if $\E[|\p_0|^2]$ is finite, we might prefer to return to \eqref{simpleA} and compute instead the sum:
\begin{equation}
\label{simpleBprime}
\frac{1}{r}  \sum_{x \in \Z, |x| \leq 100r, d \geq \bd} \E \left[ |\p_0| \1_{|\p_0| \geq \frac{d}{2}}\right] \preceq \frac{1}{r}   \E \left[ |\p_0| \sum_{d = \bd}^{100r} \1_{|\p_0| \geq \frac{d}{2}} \right]  \preceq \frac{1}{r} \E \left[ |\p_0|^2 \right] = \O(r^{-1}).
\end{equation}

In summary, if $\E[|\p_0|]$ is finite, we can choose $1 \ll \bd \ll r$ in \eqref{prochebord}, \eqref{simpleB} and see that the term $B$ goes to $0$ as $r \to \infty$, so $\PSL$ is hyperuniform. If moreover $\E[|\p_0|^2]$ is finite then we can choose $\bd = 1$ in \eqref{prochebord}, \eqref{simpleBprime} and we find that $B = \O(r^{-1})$ thus $\PSL$ is class-I hyperuniform. 
\end{proof}

\subsection*{Acknowledgement}
\

\emph{DD and DF were supported in part by the Labex CEMPI (ANR-11-LABX-0007-01), the ANR project RANDOM (ANR-19-CE24-0014) and by the CNRS GdR 3477 GeoSto.}

\emph{DF is supported by the
Czech Science Foundation, project no. 22-15763S and thanks the University of Lille for its hospitality.} 

\emph{MH is supported by the Deutsche Forschungsgemeinschaft (DFG, German Research Foundation) through the SPP 2265 \textit{ Random Geometric Systems} as well as under Germany's Excellence Strategy EXC 2044 -390685587, Mathematics M\"unster: Dynamics--Geometry--Structure.}

\emph{TL acknowledges the support of JCJC grant ANR-21-CE40-0009 from Agence Nationale de la Recherche, as well as the hospitality of Mathematics Münster, where this project was initiated.}

\bibliographystyle{alpha}
\bibliography{HUPL}

\end{document}